\documentclass[reqno,centertags, 11pt]{amsart}
\usepackage{comment,graphicx,url}
\usepackage{color}

\definecolor{purple}{rgb}{0.65, 0, 0.9}
\definecolor{orange}{rgb}{1,.5,0}
\definecolor{gray}{rgb}{0.7,.7,0.7}
\usepackage{enumitem}
\usepackage[]{hyperref}
\usepackage{mathtools}
\usepackage{esint}

\usepackage{booktabs} 
\usepackage{array} 
\usepackage{paralist} 
\usepackage{verbatim} 
\usepackage{subfig} 
\usepackage{mathrsfs}
\usepackage{amssymb}
\usepackage{amsthm}
\usepackage{amsmath,amsfonts,amssymb,esint,hyperref}
\usepackage[noabbrev, capitalize]{cleveref}
\usepackage{graphics,color}
\usepackage{enumerate, enumitem}
\usepackage{mathtools,centernot}
\usepackage{cases}
\usepackage{amsrefs}
\usepackage{bbm}
\usepackage{xfrac}
\usepackage{xcolor}

\usepackage{graphics}
\usepackage{amssymb,amsmath,amsfonts}
-\marginparwidth \oddsidemargin -9mm \evensidemargin \oddsidemargin
\usepackage{latexsym}
\advance\hoffset by 5mm

\def\@abssec#1{\vspace{.1in}\footnotesize \parindent .2in
{\bf #1. }\ignorespaces}
\newtheorem{theorem}{Theorem}[section]

\newtheorem{lemma}[theorem]{Lemma}
\newtheorem{proposition}[theorem]{Proposition}
\newtheorem{corollary}[theorem]{Corollary}

\newtheorem{remark}[theorem]{Remark}

\newcommand*{\rom}[1]{\expandafter\@slowromancap\romannumeral #1@}

\newcommand{\be}{\mathbf e}

\allowdisplaybreaks \numberwithin{equation}{section}

\renewcommand{\be}{\begin{equation}}
\newcommand{\ee}{\end{equation}}

\newcommand{\colvec}[1]{\begin{pmatrix} #1 \end{pmatrix}}

\begin{document}

\title[Stability analysis]{Stability analysis of the incompressible porous media equation and the Stokes transport system via energy structure}

\author{Jaemin Park}
\address{Department Mathematik Und Informatik, Universit\"at Basel, CH-4051 Basel, Switzerland}
\email{jaemin.park@unibas.ch}

\thanks{\textit{Acknowledgements}. The author acknowledges  the support of the SNSF Ambizione grant No. 216083. The author also extends gratitude to Luis Mart\'inez--Zoroa and Yao Yao for fruitful conversations during the course of this research.
}

\begin{abstract}
In this paper, we revisit asymptotic stability for the two-dimensional incompressible porous media equation and the Stokes transport system in a periodic channel. It is well-known that a stratified density, which strictly decreases in the vertical direction, is asymptotically stable under sufficiently small and smooth perturbations. We provide improvements in the regularity assumptions on the perturbation and in the convergence rate. Unlike the standard approach for stability analysis relying on linearized equations, we directly address the nonlinear problem by exploiting the energy structure of each system. While it is widely known that the potential energy is a Lyapunov functional in both systems, our key observation is that the second derivative of the potential energy reveals a (degenerate) coercive structure, which arises from the fact that the solution converges to the minimizer of the energy.
\end{abstract}

\maketitle
\setcounter{tocdepth}{1}

\tableofcontents

\section{Introduction}
 In this paper, we investigate asymptotic stability in the incompressible porous media equation (IPM) and the Stokes transport system in a periodic channel $\Omega:=\mathbb{T}^2\times (0,1)$. To introduce the models, let us consider a  continuity equation with a velocity field $u(t,x)$,
 \[
 \rho_t +\nabla \cdot (u \rho)=0,\text{ in $\Omega$ and $\rho(0,x)=\rho_0(x)$,}
 \]
 for some nonnegative scalar-valued function $\rho_0$, which will be referred to as density throughout the paper. Neglecting physical constants, the velocity $u$ in each model is determined by $\rho$ via 
 \[
 \begin{cases}
  u= -\nabla p -\colvec{0 \\ \rho},\text{ with } \nabla\cdot u = 0 \text{ in $\Omega$ and } u\cdot \vec{n}=0 \text{ on $\partial\Omega$,} & \text{(IPM) }\\
   \Delta u= -\nabla  p -\colvec{0 \\ \rho},\text{ with } \nabla\cdot u \text{ in $\Omega$ and } u=0 \text{ on $\partial \Omega$}, &  \text{(Stokes) }
  \end{cases}
 \]
 where  $\vec{n}$ denotes the unit normal vector on $\partial\Omega$.   We note that, given $\rho$, the above equations  (referred to as Darcy's law and the steady Stokes equation for the IPM and the Stokes transport system, respectively) uniquely determine the velocity $u$ and the pressure $p$ (up to a constant) (See e.g. \cite[Chapter 1]{constantin1988navier} for the IPM and \cite[Chapter IV]{galdi2011introduction} for the Stokes).  
 
 Both the IPM equation and the Stokes transport system describe the evolution of transported density driven by an incompressible fluid subject to gravity. Depending on the physical context, these equations can be interpreted in various ways. However, as our primary focus is on their mathematical analysis, interested readers are directed to \cite{ingham1998transport,hofer2018sedimentation,mecherbet2018sedimentation} for a more detailed exploration of the physical motivations.

 In both equations, it is well-known that any stratified density, $\rho_s(x)=\rho_s(x_1,x_2)=\rho_s(x_2)$, which is independent of the horizontal variable $x_1\in\mathbb{T}$, is a steady state. Indeed, the vector field $(0,\rho_s(x_2))^{T}$ can be easily represented as a gradient field. Therefore the velocity determined by either Darcy's law or the steady Stokes equation corresponds to a trivial vector field $u\equiv 0$. In the paper, our stability analysis will be focused on the stratified densities satisfying the following additional condition:
 \begin{align}\label{special_density}
 \partial_2\rho_s(x_2)<0,\text{ for all $x\in \Omega$}.
 \end{align}
 In the remainder of this section, we will provide a brief overview of relevant background information and present our main results  separately for each equation.
  \subsection{The IPM equation}
 We recall that the IPM equation: For $(t,x)\in \mathbb{R}^+\times \Omega$
 \begin{align}\label{IPM}
 \text{(IPM)}=\begin{cases}
\rho_t + u\cdot\nabla \rho = 0,\\
u=-\nabla p - \colvec{0\\ \rho},\quad \nabla \cdot u =0,
\end{cases}
 \end{align}
 with boundary condition $u\cdot\vec{n}=0$ on $\partial\Omega$ and $\rho(0,x)=\rho_0(x)$.
 Thanks to the incompressibility, the velocity can be recovered in terms of the stream function $\Psi$:
\begin{align}\label{IPM_stream}
u=\nabla^\perp \Psi,\text{ where $\Psi$ is a solution to }
\begin{cases}
-\Delta \Psi = \partial_1\rho & \text{ in $\Omega$,}\\
\Psi = 0 & \text{ on $\partial \Omega$.}
\end{cases}
\end{align}
  
   We first review well-posedness theory. For strong solutions, the local well-posedness with  smooth initial data for the IPM equation \eqref{special_density} can be derived by a standard energy estimate in the case where the spatial domain $\Omega$ does not have a boundary. More precisely, when $\Omega=\mathbb{T}^2$ or $\mathbb{R}^2$, any initial data in $H^{k}$ for $k>2$ possesses a unique local-in-time solution $\rho\in C([0,T] ; H^k(\Omega))$ (see \cite{cordoba2007analytical,kiselev2023small}). However, when the boundary is present, the energy estimate becomes more involved. Especially it was proved in \cite{castro2019global} that the IPM equation in a periodic channel is locally wellposed in a subspace of $H^k(\Omega)$ for $k\ge 3$ with an additional conditions on the boundary behavior of the solution (See Theorem~\ref{lwp} for a more precise statement).  Unlike the local well-posedness, the problem of  global existence v.s. finite time blow-up for smooth initial data still remains open, while several  blow up criteria have been established in \cite{cordoba2007analytical}.  For weak solutions, the global existence of the weak solutions with $L^p$ initial data is currently not available to the best of our knowledge. The nonuniqueness in the class $L^\infty([0,T]\times \mathbb{T}^2)$ by means of the convex integration \cite{szekelyhidi2012relaxation}. We also note that the global existence of patch-type solutions (the so-called Muskat problem) with the presence of a surface tension have been established in \cite{jacobs2021weak} using the optimal transport theory. 
    
    While the global well-posedness for general initial data (sufficiently smooth) is still out of reach, there are several stability results near a particular steady state $\rho_s(x_2):=1-x_2$, which, as a byproduct, ensures the global existence. Considering the IPM equation in the spatial domain $\Omega=\mathbb{R}^2$,  it was proved in  \cite{elgindi2017asymptotic} that if $\rVert \rho_0-\rho_s\rVert_{H^k(\mathbb{R}^2)}\le \epsilon$ with $k\ge 20$, the solution converges  eventually to $\rho_s$, satisfying $\rVert \rho(t)-\rho_s\rVert_{H^3(\mathbb{R}^2)}\le C\frac{\epsilon}{(1+t)^{1/4}}$. In the same paper, the author also established asymptotic stability in $\Omega=\mathbb{T}^2$ in the class $H^{k}$ for $k\ge 20$ and proved that such perturbed solutions eventually converge to a stratified density, which might not be the same  as $\rho_s$. Stability  in a periodic channel $\Omega=\mathbb{T}\times (0,1)$ was investigated in \cite{castro2019global}, proving that  in the class $H^k(\Omega)$ for $k\ge 10$ a perturbed solution eventually converges to a stratified density, which again might not be the same as $\rho_s$. While these results requires sufficiently large regularity on the initial perturbation, the authors in \cite{bianchini2024relaxation} recently reduced such regularity assumptions, establishing asymptotic stability in the class $H^k(\mathbb{R}^2)\cap H^{1-s}(\mathbb{R}^2)$ for $k>3$ and $0<s<1$. We also mention the work \cite{jo2024quantitative} which proves asymptotic stability in a periodic channel for a perturbation in $H^{k}$ for $3 <k \in\mathbb{N}$, assuming that the vertical derivative of the steady state $\rho_s$ is sufficiently large, depending on the size of the perturbation.
    
   These results concerning  asymptotic stability reveal the difficulty to specify a permanent description of the long time limit of the solution. The main challenge stems from the fact that the equilibria of the IPM are not isolated; given two stratified densities $\rho_s(x_2)$ $f(x_2)$, any function of the form $\rho_s(x_2)+\epsilon f(x_2)$ for any $\epsilon>0$ is another equilibrium. However, as noted in \cite{dalibard2023long} the IPM equation is a transport equation with an incompressible flow, therefore as long as such a limit is achieved as a strong limit in $C^1$, each super-level set of the solution $\rho(t)$ must preserve its topological properties and the area. Given an initial data, a stratified density preserving such properties can be uniquely determined as the so-called vertical (decreasing) rearrangement $\rho^*$,
   \[
   \rho(x)\mapsto \rho^*(x):=\int_0^\infty 1_{\left\{0\le x_2 \le |\left\{ \rho > s\right\}|\right\}}ds.
   \]
 Note that such vertical rearrangement is invariant under any measure preserving continuous diffeomorphism, especially, $\rho(t)^*=\rho_0^*$.  From this property, one can easily deduce that the long-time limit that was not specified in the above earlier works must be indeed $\rho_0^*$, which can be precisely specified from the initial data.
 
  A somewhat trivial but crucial observation is that $\rho^*$ is a local minimizer of the potential energy defined as
  \[
  E_P(\rho):=\int_{\Omega}\rho(x)x_2dx.
  \]
  More precisely, one can easily deduce that for any measure preserving diffeomorphism $h:\Omega\mapsto \Omega$, it holds that
  \[
  E_P(\rho^*\circ h)\ge E_P(\rho^*),
  \]  
  and the equality is obtained if and only if $h$ is the identity map.  Moreover, the potential energy is a Lyapunov functional to the IPM equation in the sense that given a sufficiently smooth solution $\rho(t)$, namely,
  \[
  \frac{d}{dt}E_P(\rho(t)) = \int_\Omega u_2\rho dx =  \int_{\Omega}u\cdot \colvec{0 \\ \rho}dx = -\int_{\Omega}|u|^2 dx,
  \]
  where the last equality is due to Darcy's law and the incompressibility of the flow. In this view, the solution to the IPM equation can be thought of as a minimizing curve associated to the potential energy whose long-time limit is the ground state of the energy. Hence the asymptotic stability can be  obtained by achieving a sufficiently fast decay of $E_P(\rho(t))$ towards $E_P(\rho_0^*)$. Our precise statement of the main theorem for the IPM equation is as follows:
\begin{theorem}\label{main_IPM}
Let $k\in \mathbb{N}$ be such that $k\ge 3$ and $\rho_s(x_2):=1-x_2$. There exist $\epsilon=\epsilon(k)>0$ and $C=C(k)>0$ such that  if  $\rho_0-{\rho}_s\in H^k_0(\Omega)$ and $\rVert \rho_0-{\rho}_s\rVert_{H^k(\Omega)}\le \epsilon$, then there exists a unique solution $\rho\in C([0,\infty); H^k(\Omega))$ to the IPM equation \eqref{IPM} and it satisfies
\[
\rVert \rho(t)-{\rho}_s\rVert_{H^k(\Omega)}\le C\epsilon,\text{ for all $t>0$},
\]
Furthermore, the potential energy decays as 
\[
E_P(\rho(t))-E_P(\rho_0^*)\le C\epsilon^2 t^{-k}.
\] Consequently, the solution $\rho(t)$ converges to the vertical rearrangement of the initial data:
\[
\rVert \rho(t)-\rho_0^*\rVert_{L^2(\Omega)}\le C \frac{\epsilon}{t^{k/2}},\text{ for all $t>0$.}
\]
\end{theorem}
A few remarks are in order: 

\begin{remark}
As mentioned earlier, several results concerning asymptotic stability for the IPM equation are available in the literature (e.g., \cites{bianchini2024relaxation,elgindi2017asymptotic,castro2019global}).  Compared to these results, our theorem requires slightly weaker regularity on the initial perturbation. Furthermore, our poof in this paper is essentially different from all of these results in that we do not rely on the linearized system but exploit the decay of the potential energy. The key observation of this paper is that the potential energy is not only a Lyapunov functional but  also reveals another (degenerate) coercive structure in its second derivative, i.e., $\left(\frac{d}{dt}\right)^2 E_P(\rho(t))\ge \rVert u_2\rVert_{L^2}^2$. See Subsection~\ref{sketch} for more detailed explanation how such property can be used in the proof. 
\end{remark}

\begin{remark}\label{remark1_regularity}
The author expects that the regularity assumption  $H^{k}$ for $k\ge3$ can be relaxed even further, and a similar strategy used this paper would work out for initial perturbation $\rho_0-\rho_s$ sufficiently small $H^{k}$ for $k> 1+\sqrt{3}$. More concrete evidence of such an expectation will be explained in Subsection~\ref{sketch}. The main reasons why we do not included such a stronger statement are because we do not want to introduce extra complications of the proof by involving fractional Sobolev spaces, and the local well-posedness of the equation is currently not directly available for lower regularity spaces (see Theorem~\ref{lwp}).  We emphasize that the perturbation regularity cannot be relaxed too much, considering the long time instability result in $H^{2-\epsilon}$ for any $\epsilon>0$ \cite[Theorem 1.5]{kiselev2023small}.  Noticing that $H^{2}$ barely fails to embed into $C^1$, it seems to be an interesting question whether an asymptotic stability can be established for a small perturbation in $H^{2+\epsilon}$.
\end{remark}

\begin{remark}\label{kinetic_statemt}
From the stream function formulation \eqref{IPM_stream}, it is straightforward to see the velocity and the density formally stay in the same regularity class, especially,  $\rVert u\rVert_{L^2}\le \rVert \rho-\overline{\rho}\rVert_{L^2}$, where $\overline{\rho}$ is the spatial average of $\rho$. However, compared to the decay rate of the density stated in the above theorem, our argument reveals slightly better decay rates of the kinetic energy and the anisotropic kinetic energy  in a time-average sense, that is,
\[
\frac{2}{t}\int_{t/2}^t \rVert u(s)\rVert_{L^2}^2ds \le C_k\frac{\epsilon^2}{t^{k+1}},\text{ and } \frac{2}{t}\int_{t/2}^{t}\rVert u_2(s)\rVert_{L^2}^2 ds\le C_k\frac{\epsilon^2}{t^{k+2}}, \text{ for all $t>0$, $k\ge 3$.}
\]
\end{remark}

\begin{remark}\label{remark2_regularity}
Theorem~\ref{main_IPM} only concerns  perturbations near a specific steady state $\rho_s(x_2)=1-x_2$. However, the author expects that our result can be generalized to perturbations near any sufficiently regular stratified density $\rho_s$ such that $\inf_{x\in \Omega}\partial_2\rho_s(x_2)>0$. Again the reason why we do not pursue such a more general statement is due to a lack of an exact statement in the literature about a local well-posedness theorem  near general  steady states. Instead, we will establish  asymptotic stability near  general stratified steady states for the Stokes transport system (see Theorem~\ref{main_Stokes}), which could  provide more evident ideas for the IPM equation as well. 
\end{remark}
\subsection{Stokes transport system}
 The Stokes transport system is another active scalar equation and it  shares several interesting properties with the IPM equation. We first recall the system:
\begin{align}\label{Stokes}
\text{(Stokes)} = \begin{cases}
\rho_t + u\cdot\nabla \rho = 0, &\text{for $(t,x)\in \mathbb{R}^+\times \Omega$}\\
\nabla \cdot u =0,\\
\Delta u= -\nabla  p -\colvec{0 \\ \rho}, \quad u=0 &\text{ on $\partial \Omega$}.
\end{cases}
\end{align}
As in the IPM, the incompressibility condition allows for a stream function formulation for the velocity field:
\begin{align}\label{Stokes_stream}
u=\nabla^\perp \Psi,\text{ where $\Psi$ solves }
\begin{cases}
\Delta^2 \Psi = \partial_1\rho & \text{ in $\Omega$,}\\
\Psi =\nabla\Psi =0 & \text{ on $\partial \Omega$.}
\end{cases}
\end{align}
 From the stream function formulation, one can easily notice that the velocity in the Stokes transport is much more regular than in the IPM equation. Indeed, such regular structure enables a global well-posedness theorem in a standard manner; if $\rho_0\in H^k(\Omega)$ for $k\ge 3$, then there exists a unique solution $\rho\in C([0,\infty), H^k(\Omega))$ (see, \cite[Theorem A.1]{dalibard2023long} or \cite[Theorem 1.1]{leblond2022well}).  We also mention that the regular of the velocity can yield quite robust structures for  weak solutions, for instance,  $L^3$ initial data $\rho_0$ yields a unique regular Lagrangian solution for three-dimensional model \cite[Theorem 2.2, Theorem 2.4]{inversi2023lagrangian}. 
 
   Various long time behaviors of the system \eqref{Stokes} have been investigated in \cite{dalibard2023long}, where the authors studied asymptotic stability and boundary layer formation for initial data near $\rho_s(x_2)=1-x_2$. We also mention that the authors in \cite{gancedo2022long,gancedo2024global} studied the interface problem, where $\rho$ is given as a characteristic function representing two different fluid densities, establishing global existence and asymptotic stability/instability of the interface depending on the Reyleigh-Taylor stability criterion.
 
 In regard to the Stokes transport system, our main result in this paper is a slight extension of the asymptotic stability obtained in \cite{dalibard2023long}, especially concerning the regularity assumption on the initial perturbation. This result will be established by adapting a similar strategy that we exploit for the IPM equation based on the energy structure. More precise statement is as follows:
\begin{theorem}\label{main_Stokes}
Let $\rho_s(x_2)$ be a stratified steady state such that
\begin{align}\label{steady_state_assumption1}
\gamma:=\inf_{x\in \Omega}\left(-\partial_2\rho_s(x_2)\right) >0,\quad \rVert \rho_s\rVert_{H^4(\Omega)} <\infty.
\end{align}
Then there exist $\epsilon=\epsilon(\gamma,\rVert \rho_s\rVert_{H^4(\Omega)})$ and $C=C(\gamma,\rVert \rho_s\rVert_{H^4})$ such that 
if $\rho_0-{\rho}_s\in H^2_0(\Omega)\cap H^4(\Omega)$ and $\rVert \rho_0-{\rho}_s\rVert_{H^4}\le \epsilon$, then the unique solution $\rho\in C([0,\infty); H^4(\Omega))$ to the Stokes transport system \eqref{Stokes} satisfies  
\[
\rVert \rho(t)-{\rho}_s\rVert_{H^4(\Omega)}\le C\epsilon,\text{ for all $t>0$}.
\] Furthermore, the potential energy decays as 
\[
E_P(\rho(t))-E_P(\rho_0^*)\le C\epsilon^2 t^{-2}.
\] Consequently, the solution $\rho(t)$ converges to its vertical rearrangement:
\[
\rVert \rho(t)-\rho_0^*\rVert_{L^2(\Omega)}\le C\frac{\epsilon}{t},\text{ for all $t>0$.}
\]
\end{theorem}
\begin{remark}
A similar asymptotic stability result was already provided in \cite[Theorem 1.1]{dalibard2023long}, where the authors assumed that the initial perturbation is  small in $H^{6}(\Omega)$ near $\rho_s(x_2)=1-x_2$. The authors also provided clear evidence that such a result can be obtained near more general stratified steady state which is sufficiently regular. As noted earlier, the proof presented in this paper is different in that our analysis does not use the linearized equation and it is an energy functional based method. This approach requires  a slightly weaker regularity assumption for the initial perturbation. However, as in the IPM equation, there is a threshold of the regularity for the stability. Indeed, \cite[Theorem 3.7.2]{leblond2023well} provides an example of long time instability for small initial perturbation in $H^{2-\epsilon}$ near any steady state. 
\end{remark}
\begin{remark} As  in Remark~\ref{kinetic_statemt}, our proof reveals a slight faster decay estimate for the velocity compared to that of the density. More precisely, we obtain
\[
\frac{2}{t}\int_{t/2}^t \rVert\nabla u(s)\rVert_{L^2}^2ds \le C\frac{\epsilon^2}{t^{3}},\text{ and } \frac{2}{t}\int_{t/2}^{t}\rVert u_2(s)\rVert_{L^2}^2 ds\le C\frac{\epsilon^2}{t^{4}}, \text{ for all $t>0$.}
\]
\end{remark}

\subsection{A sketch of the proof of Theorem~\ref{main_IPM}}\label{sketch}
Let us  describe the structure of the proof for the IPM equation.  A similar strategy will be adapted to prove the stability for the Stokes transport system. 

 We consider initial data $\rho_0$ such that $\rVert \rho_0-\overline{\rho}\rVert_{H^k}\le \epsilon$ for sufficiently small $\epsilon>0$ and we will assume $k \ge 3$. For convenience, we denote
\[
\theta(t):= \rho(t)-{\rho}_s,\quad 
\rho_0^*:=\text{ the vertical (decreasing) rearrangement of $\rho_0$,}
\]
and 
\[
E(t):=\int_{\Omega} (\rho(t)-\rho_0^*)(x)x_2 dx,\quad K(t):=\rVert \nabla \Psi(t)\rVert_{L^2}^2.
\]
Thanks to the weight, $x_2$, in the integral expression for $E$, it is evident that $\rho_0^*$ is the unique minimizer of $\rho\mapsto E_P(\rho)$ among all the functions which can be obtained by a pushforward of $\rho_0$ under a measure preserving map. Also if $\rho_0^*$ is a non-degenerate minimizer, then $\rho\mapsto E_P(\rho)$ is expected to satisfy a quadratic lower bound  in a suitable space. In this paper, we will look for such a lower bound in $L^2(\Omega)$ and establish in Proposition~\ref{propoos} that
\begin{align}\label{rk1psd1}
E(t)\sim \rVert \rho(t)- \rho_0^*\rVert_{L^2}^2.
\end{align}
Furthermore, as long as the solution $\rho(t)$ stays close to a stratified density in the space $H^3(\Omega)$, the fact that $H^3(\Omega)$ continuously embeds into $C^1(\Omega)$ suggests that each level set of $\rho(t)$ is also a graph of the horizontal variable $x_1$, from which one can infer that 
\begin{align}\label{rk1psd12}
\rVert \rho(t)-\rho_0^*\rVert_{L^2}\le C\rVert \partial_1\rho(t)\rVert_{L^2}.
\end{align}
On the other hand, the time derivative of $E(t)$ can be computed as
\begin{align}\label{rkosdsd22}
\frac{d}{dt}E(t)=\frac{d}{dt}\int \rho(t,x)x_2dx = \int u_2 \rho dx = \int \partial_1\Psi\rho dx = -\int \Psi \partial_1 \rho dx = \int \Psi \Delta \Psi dx = - K(t).
\end{align}
The Biot-Savart law in \eqref{IPM_stream} and the Gargliado-Nirenberg inequality tell us that 
\[
 \rVert \partial_1\rho\rVert_{L^2} =\rVert \Delta\Psi\rVert_{L^2}\le C_k \rVert \nabla\Psi\rVert_{H^k(\Omega)}^{1/k}\rVert \nabla \Psi\rVert_{L^2}^{(k-1)/k}.
\]
Combining this with \eqref{rk1psd1} and \eqref{rk1psd12}, we get
\[
K(t)=\rVert \nabla \Psi\rVert_{L^2}^2 \ge_C \rVert \partial_1\rho\rVert_{L^2}^{2k/(k-1)}\rVert \nabla \Psi\rVert_{H^k}^{-2/k}\ge_C E(t)^{k/(k-1)}\rVert \nabla\Psi\rVert_{H^k}^{-2/(k-1)}.
\]
Substituting this into \eqref{rkosdsd22}, we obtain
\begin{align}\label{rjksd1xc}
\frac{d}{dt}E(t) \le_C -E(t)^{k/(k-1)}\rVert \nabla\Psi(t)\rVert_{H^k}^{-2/(k-1)}.
\end{align}
This inequality is the main source of the asymptotic stability.  Let us use the following notation  which is slightly different from the usual convention: For $\alpha>0$ and $f:\mathbb{R}^+\mapsto \mathbb{R}^+$,
\begin{align}\label{rkjsdwdsdasx1}
f(t) = O(t^{-\alpha}),\text{ if \  $\frac{2}{t}\int_{t/2}^t f(s)ds\le C(1+t)^{-\alpha}$ for some $C>0$}.
\end{align}
Clearly $f(t)=O(t^{-\alpha})$ means that $f$ decays like $t^{-\alpha}$ in average. Let us make an ansatz:
\begin{align}\label{anzasd1}
\rVert \theta(t)\rVert_{H^k}\lesssim \epsilon,\quad \rVert \nabla\Psi(t)\rVert_{H^k} = O(t^{-\alpha}), \text{ for all $t>0$, for some $\alpha>0$} 
\end{align}
Under this ansatz, one can immediately deduce from the inequality \eqref{rjksd1xc} that 
\begin{align}\label{pot1}
E(t)= O({t^{-(k+2\alpha-1)}}).
\end{align}
With this energy decay rate, the energy variation in time \eqref{rkosdsd22} suggests that $K(t)$ decays faster than $E(t)$ by a factor of $t^{-1}$. Indeed, this elementary heuristic can be made rigorous by measuring the decay rates as an average (Lemma~\ref{ODEBDC}). Hence we can deduce 
\[
K(t) =  O(t^{-(k+2\alpha)}).
\]
  Having such a decay rate for $K(t)$, we will proceed to look at a higher derivative of the potential energy in time.  A key observation  is that the second derivative of the energy $E(t)$ also exhibits a coercive structure, namely, 
  \[
 \left(\frac{d}{dt}\right)^2 E(t) = -\frac{d}{dt}K(t) \ge C \rVert u_2(t)\rVert_{L^2}^2,
 \]
 which is the result of Proposition~\ref{main_IPM_decay}. We emphasize that such  coercive structure should not come as a surprise,  because the solution is expected  to converge to  a non-degenerate minimizer of the potential energy.  Again, our notation \eqref{rkjsdwdsdasx1} allows us to postulate that  $\rVert u_2\rVert_{L^2}^2$ will decay faster than $K(t)$ by a factor $t^{-1}$, that is,
 \begin{align}\label{rjjsd1ssxxsdsd}
 \rVert u_2(t)\rVert_{L^2}^2 = O(t^{-(k+2\alpha+1)}).
 \end{align}
So far, the decay rates of the energies have been derived under the ansatz \eqref{anzasd1}, therefore it must be justified in order to close the argument. To this end, in Proposition~\ref{energy_IPM_estimate}, we will derive the following estimate (an informal estimate is presented at this point for simplicity):
  \begin{align}\label{whatamidoingnow}
  \frac{d}{dt}\rVert \theta(t)\rVert_{H^k}^2 + \rVert\nabla \Psi(t)\rVert_{H^k}^{2}\le C \rVert u_2\rVert_{W^{1,\infty}}\rVert \theta(t)\rVert_{H^k}^2.
  \end{align}
 Using again the Gagliardo-Nirenberg interpolation inequality and Young's inequality,  we deduce 
 \[
 \rVert u_2\rVert_{W^{1,\infty}}\le C\rVert u_2\rVert_{L^2}^{1-2/k}\rVert u_2\rVert_{H^k}^{2/k}\le C\rVert u_2\rVert_{L^2}^{1-2/k}\rVert \nabla\Psi\rVert_{H^k}^{2/k}\le \eta \rVert \nabla\Psi\rVert_{H^k}^2 + C_\eta\rVert u_2\rVert_{L^2}^{(k-2)/(k-1)},
 \]
 for any $\eta\ll 1$. Noting that $\rVert \theta(t)\rVert_{H^k}\lesssim \epsilon\ll 1$ as long as the ansatz \eqref{anzasd1} is valid, we  substitute this  estimate  into \eqref{whatamidoingnow}, yielding that \begin{align}\label{differwhatdoyouwant}
  \frac{d}{dt}\rVert \theta(t)\rVert_{H^k}^2 + \rVert\nabla \Psi(t)\rVert_{H^k}^{2}\le C\rVert u_2\rVert_{L^2}^{(k-2)/({k-1})}\rVert \theta(t)\rVert_{H^k}^2= \epsilon^2O( t^{-(k+2\alpha+1)(k-2)/(2k-2)}),
 \end{align}
 where the last equality is due to \eqref{rjjsd1ssxxsdsd}. In this differential inequality, a sufficient condition for the ansatz \eqref{anzasd1} to persist is that  the right-hand side should decay  faster than $O(t^{-(1)})$, that is,
 \begin{align}\label{ksufficient}
 \frac{(k+2\alpha+1)(k-2)}{2k-2} > 1.
 \end{align}
 Indeed, if this condition is satisfied, integrating the both sides of \eqref{differwhatdoyouwant}  in time yields that
 \[
 \sup_{t>0}\rVert \theta(t)\rVert_{H^k}^2 + \int_0^{\infty}\rVert \nabla \Psi(t)\rVert_{H^k}^2dt \le C_{k,\alpha,\theta_0} \epsilon^2.
 \]
 In this case, $t\mapsto \rVert \nabla \Psi(t)\rVert_{H^k}^2$ is integrable in time, which indicates that  our ansatz \eqref{anzasd1} should hold at least for some $\alpha\ge 1/2$; 
 \[
 \frac{2}{t}\int_{t/2}^{t}\rVert \nabla \Psi(s)\rVert_{H^k}ds\le \left(\frac{2}{t}\int_{t/2}^t \rVert \nabla \Psi(s)\rVert_{H^k}^2ds\right)^{1/2}\le C_{k,\alpha,\theta_0}\epsilon\sqrt{\frac{2}{t}} \text{ for all $t>0$.}
 \]
  For  $\alpha\ge 1/2$, the minimum value of $k$ for the sufficient condition \eqref{ksufficient} to hold can be directly computed:
 \[
\frac{(k+2\alpha+1)(k-2)}{2k-2}>1 \Longleftarrow \frac{(k+2)(k-2)}{2k-2} > 1\Longleftarrow  k> 1+\sqrt{3}.
 \]
 The range of $k$ stated in Theorem~\ref{main_IPM} is strong enough to satisfy the sufficient condition for the above scheme. Especially, \eqref{pot1} with $\alpha\ge1/2$ directly gives the decay rate of the  potential energy stated in the Theorem~\ref{main_IPM}, resulting in the desired asymptotic stability.

 \subsection{Organization of the paper}
 In Section~\ref{preliminearty}, we collect useful tools concerning simple ODE problems and quantitative estimates for the potential energy. The stability analysis for the IPM equation and the Stokes transport system will be separately investigated in Section~\ref{IPMsection} and Section~\ref{Stokessection} respectively.

\subsection{Conventional notations}
 Following the conventional practice, we denote by $C$ an implicit positive constant that may vary from line to line. In the case where $C$ depends on a quantity,  say $A$, we will represent it  as $C_A$ or $C(A)$. For two quantities, $A$ and $B$, we will also use the notation $A \le_C B$, indicating that $A \le CB$ for some constant $C > 0$.
We denote 
\begin{align}\label{cc_comv}
C^\infty_0(\Omega):=\left\{ f\in C^\infty(\Omega): {\text{supp}(f)}\subset \Omega\right\},
\end{align}
where $\text{supp}(f)$ is the closed support of $f$.

 

\section{Preparation: Time-average decay and vertical rearrangement}\label{preliminearty}
\subsection{Time-average decay rates in differential inequalities}
In the proof of asymptotic stability, we will measure the decay rates of the energy quantities in a time-average manner. To prepare for this analysis, we will gather useful lemmas concerning simple differential inequalities. In what follows $[0,T]$ will denote an arbitrary time interval for some $T>0$.
\begin{lemma}\label{ABC_ODE}
Let $\alpha>0,\ 1<n$. Let  $a(t)$ and $f(t)$ be nonnegative functions on $[0,T]$ such that
\[
\frac{d}{dt}f(t) \le- a(t)^{-\alpha}f(t)^n,\quad f(0)=f_0.
\]
Then, $f$ satisfies
\[
f(t)\le_{\alpha,n} \frac{A^{\alpha/(n-1)}}{t^{(\alpha+1)/(n-1)}},\text{for all $t\in [0,T]$, where $A:=\int_0^t a(s)ds.$}
\]
\end{lemma}
\begin{proof}
Dividing the differential inequality by $f(t)^n$ and integrating it in time, we find that
\begin{align}\label{exact1}
\frac{1}{f^{n-1}(t)}-\frac{1}{f_0^{n-1}} \ge_{n} \int_0^t a(s)^{-\alpha}ds.
\end{align}
Since $\alpha>0$, Jensen's inequality yields that $
\int_0^t a(s)^{-\alpha}\frac{ds}{t}  \ge \left(\int_0^t a(s)ds\right)^{-\alpha}t^\alpha= A^{-\alpha}t^{\alpha},$
which implies
\[
\int_0^t a(s)^{-\alpha}ds \ge_n A^{-\alpha}t^{\alpha+1}.
\]
Plugging this into \eqref{exact1}, we get $
\frac{1}{f^{n-1}(t)}\ge_n A^{-\alpha}t^{\alpha+1}$, which immediately gives  the desired result.
\end{proof}

\begin{lemma}\label{ODEBDC}
Let $n>0$ and $f(t),g(t),h(t)$ be nonnegative functions on $[0,T]$  such that
\begin{align}\label{ODEsure}
\frac{d}{dt}f(t)\le-g(t),\ \frac{d}{dt}g(t)\le -h(t),\text{ and } f(t)\le \frac{C}{t^n},
\end{align}
for some $C>0$.
Then,  it holds that
\[
\frac{2}{t}\int_{t/2}^t g(s)ds\le_n \frac{C}{t^{n+1}},\text{ and } \frac{2}{t}\int_{t/2}^{t}h(s)ds\le_n \frac{C}{t^{n+2}},\text{ for all $t\in [0,T]$.}
\]
\end{lemma}
\begin{proof}
Let us choose $s,t\in[0,T]$ arbitrary so that $0\le s\le t\le T$. Integrating $f'(t)\le -g(t)$  over $[s,t]$ for $s\in (0,t)$, we get $f(t)-f(s)+\int_s^t g(u)du\le 0$. Hence, the upper bound of $f(t)$ tells us that 
\begin{align}\label{alo}
\int_s^t g(u)du \le f(s)\le \frac{C}{s^{n}},\text{ for  $0<s<t<T$.}
\end{align}
Plugging in $s=t/2$, we see that
\begin{align}\label{112se}
\frac{2}{t}\int_{t/2}^t g(u)du\le \frac{C}{t^{n+1}},
\end{align}
and this is the desired estimate for $g$. 

Similarly, we integrate $g'\le -h$ and observe that
\begin{align}\label{onlyrjskdsd1sd}
\int_{s}^t h(u)du \le g(s),\text{ for $0<s<t<T$}.
\end{align}
Integrating one more time in $s$ over $[t/4,t]$, we find that the left-hand side must satisfy
\[
\int_{t/4}^t \int_s^th(u)duds = \int_{t/4}^th(u)\int_{t/4}^udsdu =\int_{t/4}^t h(u)(u-t/4)du\ge \int_{t/2}^t h(u)(u-t/4)du\ge \frac{t}4\int_{t/2}^th(u)du.
\]
On the other hand, integrating the right-hand side in \eqref{onlyrjskdsd1sd} over $[t/4.t]$ yields 
\[
\int_{t/4}^t g(s)ds = \int_{t/4}^{t/2}g(s)ds + \int_{t/2}^t g(s)ds\le_n \frac{C}{(t/2)^{n}}+\frac{C}{t^{n}}\le_n \frac{C}{t^{n}},
\]
where we used \eqref{112se}. Putting them together, we obtain
\[
t\int_{t/2}^{t}h(u)du\le_n \frac{C}{t^{n}},
\]
Dividing the both sides by $t^2$, we derive the desired estimate for $h$, finishing the proof.
\end{proof}

 It is an elementary fact that if a bounded function $f$ exhibits  a decay rate $O(t^{-(1+\epsilon)})$, it is  integrable over all time, i.e., $\int_0^\infty f(t) dt <C_\epsilon<\infty$. In the next lemma, we demonstrate that if $f(t)$ decays like $O(t^{1+\epsilon})$ in a time-average sense, the same conclusion holds. 

\begin{lemma}\label{fialmass}
Let $T> 2$ and $n> 1$. Let $f(t)$ be a nonnegative function on $[0,T]$ such that
\[
\frac{2}{t}\int_{t/2}^{t} f(s)ds\le \frac{E}{t^n},\text{ for some $E>0$, for all $t\in [2,T]$}.
\]
Then, for $\alpha\in (1/n,1]$, we have 
\[
\int_1^T f(t)^{\alpha}ds \le C_{\alpha,n}E^{\alpha},
\]
where $C_{\alpha,n}>0$ does not depend on $T$.
\end{lemma}
\begin{proof}
We pick $N\in\mathbb{N}$ such that
\begin{align}\label{Cucuvick}
\frac{T}{2^{N+1}}\le 1\le \frac{T}{2^{N}}\le 2,
\end{align}
and define $T_{i}:={2^{-i}}T$, for $i=0,\ldots, N$. We decompose
\[
\int_1^T f(t)^{\alpha}ds  = \int_{1}^{T_N}f(t)^\alpha dt + \sum_{i=1}^{N}\int_{T_i}^{T_{i-1}} f(t)^\alpha dt.
\]
Since $T_N\le 2$ and $\alpha\le 1$, applying Jensen's inequality, we obtain 
\begin{align}\label{apj}
 \int_{1}^{T_N}f(t)^\alpha dt\le \int_1^2 f(t)^\alpha dt \le \left(\int_1^2 f(t)dt \right)^\alpha\le E^\alpha.
 \end{align}
For $1\le i\le N$,  again Jensen's inequality gives us that
\[
\int_{T_i}^{T_{i-1}} f(t)^\alpha dt \le \left(\frac{1}{T_{i-1}-T_{i}}\int_{T_i}^{T_{i-1}}f(t)dt \right)^\alpha |T_{i-1}-T_i| = \left(\frac{1}{T_i}\int_{T_i}^{T_{i-1}}f(t)dt\right)^{\alpha}T_i\le_{\alpha} E^\alpha T_i^{1-\alpha n},
\]
where the last inequality follows from the decay hypothesis for $f$. Summing over $i=1,\ldots,N$, we get
\[
\sum_{i=1}^N\int_{T_i}^{T_{i-1}} f(t)^\alpha dt\le_\alpha E^\alpha T^{1-\alpha n }\sum_{i=1}^{N}\left({2^{\alpha n-1}}\right)^{i}\le_{\alpha,n} E^\alpha\left(\frac{T}{2^N}\right)^{1-\alpha n}\le  C_{\alpha,n}E^\alpha,
\]
where the last inequality follows from \eqref{Cucuvick}. Combining this with \eqref{apj}, we finish the proof.\end{proof}

\subsection{Vertical rearrangement}
Given a Borel measurable function $f$ on $\Omega$, we define its vertical (decreasing) rearrangement as
  \begin{align}\label{vertical_rearrangement}
  f^*(x_2):=\int_0^\infty 1_{\left\{0\le x_2 \le |\left\{ f > s\right\}|\right\}}ds.
  \end{align}
  By its definition, it is clear that $x_2\mapsto f^*(x_2)$ is monotone decreasing.  In the rest of the section, we consider a stratified density $\rho_s(x)=\rho_s(x_2)$, a function $f$ that is close to $\rho_s$ in a Sobolev space and its vertical rearrangement. 
  
Before presenting the lemmas, let us collect some basic properties for $\rho_s$.  We will always assume that 
\begin{align}\label{strati_sd2}
\gamma:=\inf_{\Omega}(-\partial_2\rho(x_2))>0,\quad \rVert \rho_s\rVert_{H^4(\Omega)} < \infty.
\end{align}
By the monotonicity of $\rho_s$, we can describe the image of $\rho_s$ as
\begin{align}\label{interval_2}
I:=\rho_s(\Omega)= [\rho_s(1),\rho_s(0)].
\end{align}
The inverse function theorem, combined with the assumption that $\gamma>0$, guarantees the existence of the inverse of $\rho_s$, that is,  $\phi_0:=\rho_s^{-1}:I\mapsto [0,1]$ is well-defined. Moreover, since $\rho_s$ depends on the single variable $x_2$, the regularity assumption in \eqref{strati_sd2}, combined with the usual Sobolev embedding theorem, ensures that $\rho_s\in C^3(\Omega)$, and $\rVert \rho_s\rVert_{C^3}\le_C \rVert \rho_s\rVert_{H^4}.$ With such information, one can straightforwardly deduce the following estimates:
\begin{align}\label{rkskd2sd}
\rVert\partial_s\phi_0\rVert_{L^\infty}+\rVert \partial_{ss}\phi_0\rVert_{L^\infty} + \rVert \partial_2\rho_s\rVert_{L^\infty}+\rVert \partial_{22}\rho_s\rVert_{L^\infty}+\rVert \partial_{222}\rho_s\rVert_{L^{\infty}}\le C(\gamma,\rVert \rho_s\rVert_{H^4}).
   \end{align}
 Noting that $H^3(\Omega)$ continuously embeds into $C^1(\Omega)$, one can infer that if a function $f$ is sufficiently close to $\rho_s$ in $H^3$, similar properties of the level sets and the inverse function of $f$ can be quantitatively estimated. This is will be the main implication of the next lemma. In the rest of this section, the implicit constant $C$, that appears in the proofs, may depend on $\gamma$ and $\rVert \rho_s\rVert_{H^4}$ but we will omit its dependence in the notation for simplicity.
   
  \begin{lemma}\label{rearrangement_lem}
  Suppose $\rho_s$ satisfies \eqref{strati_sd2}. There exists $\delta=\delta(\gamma,\rVert \rho_s\rVert_{H^4})>0$, such that if 
  \[
  f={\rho}_s\text{ on $\partial\Omega$ and } \rVert f-{\rho}_s\rVert_{H^3(\Omega)}\le \delta,
  \] then there exist $\phi_1:I\mapsto [0,1]$ and $h:\mathbb{T}\times I\to [0,1]$ such that 
 \begin{align*}
\int_{\mathbb{T}}h(x_1,s)dx_1=0\text{ and } f(x_1,\phi_1(s)+h(x_1,s))= s = f^*(\phi_1(s)),\text{ for $(x_1,s)\in \mathbb{T}\times I$.}
 \end{align*}
 Furthermore, the following estimates hold:
 \begin{align}\label{regularity_level}
\rVert\partial_{ss}(\phi_1-\phi_0)\rVert_{L^\infty}+ \rVert\partial_s(\phi_1-\phi_0)\rVert_{L^\infty} + \rVert h\rVert_{L^\infty}+\rVert \partial_sh\rVert_{L^\infty}\le C \rVert f-{\rho}_s\rVert_{H^3}. 
 \end{align}
 where $C>0$ is a constant which depends on $\rVert \rho_s\rVert_{H^4}$ and $\gamma$.
  \end{lemma}

\begin{figure}
\hspace{0.3cm}\includegraphics[scale=0.8]{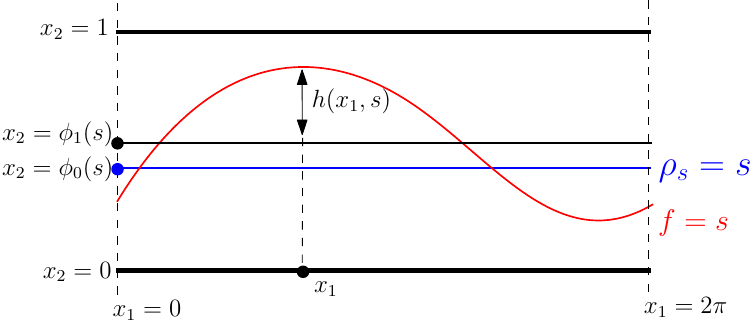}
\caption{Illustration of the level sets of $\rho_s$ and $f$. For each $s\in I$, $\phi_1$ is uniquely determined so that the deviation $x_1\mapsto h(x_1,s)$ has zero average in $x_1$.}
\label{fig1}
\end{figure}
   \begin{proof} An illustration for the proof of the lemma is presented in Figure~\ref{fig1}.  
   
    We notice that the image of $f$ is exactly equal to the interval $I$, which is the image of ${\rho}_s$. Indeed, we have
   \begin{align}\label{monotone1}
  \partial_2 f (x_1,x_2)&= \partial_2{\rho}_s(x_2) + (\partial_2 f(x_2) - \partial_2 {\rho}_s(x_2))\le  -\gamma + C\rVert f-{\rho}_s\rVert_{H^3}\le  -\gamma+C\delta < 0,
  \end{align}
  for sufficiently small $\delta>0$.
   Hence,  $x_2\mapsto f(x_1,x_2)$ is  strictly decreasing. Since $f={\rho}_s$ on $\partial\Omega$, it follows that $f(\overline{\Omega})=I$.
    
  Thanks to the monotonicity, the implicit function theorem tells us that there is a parametrization of the level curves, $\left\{f=s\right\}$, which we will denote by $\phi(\cdot,s)$, that is,
  \begin{align}\label{level_set_ff}
  f(x_1,\phi(x_1,s))=s\text{ for each $(x_1,s)\in \mathbb{T}\times I$.}
  \end{align} Let us rewrite it as
  \begin{align}\label{phi_decomposition}
  \phi(x_1,s) = \phi_0(s) + g(x_1,s), \text{ where } g(x_1,s):=\phi(x_1,s) -\phi_0(s).
   \end{align}
Now, we will aim to derive necessary estimates for  $g$. Writing
      \begin{align*}
     0  = \rho_s(\phi_0(s))-f(x_1,\phi(s)) = \rho_s(\phi_0(s)) - \rho_s(\phi(x_1,s))+ \left(\rho_s(\phi(x_1,s)) -f(x_1,\phi(x_1,s))
     \right),
     \end{align*}
     we notice that
     \begin{align}\label{betterex}
    \int_{0}^{g(x_1,s)}\partial_2\rho_s(y+\phi_0(s))dy=  \rho_s(\phi(x_1,s)) -f(x_1,\phi(x_1,s)).
     \end{align}   
   Since $\partial_2\rho_s<-\gamma<0$ and $ \rVert f-\rho_s\rVert_{L^\infty}\le_C \rVert f-\rho_s\rVert_{H^3}$, we find that 
\begin{align}\label{gsmall}
\rVert g\rVert_{L^\infty}\le_C \rVert f-\rho_s\rVert_{H^3}.
 \end{align}
Differentiating \eqref{betterex} with respect to $s$, we obtain
\begin{align}\label{tlqkfeljsd12sd}
\partial_s g(x_1,s)\partial_2\rho_s (\phi(x_1,s)) +  \int_{0}^{g(x_1,s)}\partial_{22}\rho_s(y+\phi_0(s))\partial_s\phi_0(s)dy = (\partial_2\rho_s - \partial_2f)(x_1,\phi(x_1,s))\partial_s\phi(x_1,s).
\end{align}
Again, using $\partial_2\rho_s<-\gamma<0$, we have $|\partial_s g\partial_2\rho_s |\ge C | \partial_sg|$. On the other hand, using \eqref{rkskd2sd} and \eqref{gsmall}, we can estimate  the integral on the left-hand side as
\begin{align*}
\left| \int_{0}^{g(x_1,s)}\partial_{22}\rho_s(y+\phi_0(s))\partial_s\phi_0(s)dy\right|\le_C \rVert f-\rho_s\rVert_{H^3}.
\end{align*} For the right-hand side, at each point $(x_1,s)\in\mathbb{T}\times I$,  we have
\[
|(\partial_2\rho_s - \partial_2f)\partial_s\phi|\le_C \rVert \rho_s-f\rVert_{H^3}|\partial_s\phi| \le_C \rVert \rho_s-f\rVert_{H^3} |\partial_s\phi_0| + \delta |\partial_s g|\le_C \rVert \rho_s-f\rVert_{H^3} + \delta |\partial_s g|.
\]
Assuming $\delta$ is sufficiently small depending on $\gamma$, these estimates and \eqref{tlqkfeljsd12sd} give us 
\begin{align}\label{rjjsdwdqdsd}
 \rVert\partial_s g\rVert_{L^\infty}\le_C \rVert f-\rho_s\rVert_{H^3}.
\end{align}
Next let us estimate $\partial_{ss}g$.
Once again, differentiating \eqref{tlqkfeljsd12sd} in $s$ and using the chain rule, we get (at each point $(x_1,s)$)
\begin{align}\label{rjjksdsd1sd}
\partial_{ss}g \partial_2 \rho_s(\phi)&=A_1+A_2+A_3+A_4+A_5+A_6,
\end{align}
where
\begin{align*}
A_1&:=-\partial_s g \partial_{22}\rho_s(\phi)\partial_s\phi,\\
 A_2&:=-\partial_sg\partial_{22}\rho_s(\phi)\partial_s\phi_0\\
 A_3&:=-\int_{0}^{g}\partial_{222}\rho_s(y+\phi_0)(\partial_s\phi_0)^2dy \\
 A_4&:=-\int_0^{g}\partial_{22}\rho_s(y+\phi_0)\partial_{ss}\phi_0(s)dy,\\
 A_5&:=\partial_{22}(\rho_s-f)(\phi)(\partial_s\phi)^2\\
 A_6&:=\partial_{2}(\rho_s-f)(\phi) \partial_{ss}\phi.
\end{align*}
Again, $\partial_2\rho_s\le -\gamma<0$ gives us that the left-hand side of \eqref{rjjksdsd1sd} can be estimated as
 \[
 |\partial_{ss}g\partial_2\rho_s(\phi)|\ge C |\partial_{ss}g|.
 \] Using \eqref{rkskd2sd}, \eqref{gsmall} and \eqref{rjjsdwdqdsd}, it is straightforward that
\begin{align*}
|A_1|, |A_2|, |A_3|,   |A_4| \le_C \rVert f-\rho_s\rVert_{H^3},\quad |A_5|\le_C |\partial_{22}(\rho_s-f)|,\quad  |A_6|\le_C \rVert f-\rho_s\rVert_{H^3} + \delta| \partial_{ss}g|.
\end{align*}
Plugging them into \eqref{rjjksdsd1sd}, we obtain a pointwise estimate:
\begin{align}\label{rhhsjdasd1sd}
|\partial_{ss}g(x_1,s)|& \le_C \rVert f-\rho_s\rVert_{H^3} +  |\partial_{22}\rho_s(\phi(x_1,s))-\partial_{22}f(x_1,\phi(x_1,s))|.
\end{align}
Then for each fixed $x_1$, we treat the functions above as a function of $s$.  Applying the Sobolev embedding $W^{1,1}(I)\mapsto L^\infty(I)$, we have that for each $x_1\in\mathbb{T}$,
    \begin{align}\label{sympat}
   \sup_{s\in I} &|\partial_{22}\rho_s(\phi(s)) - \partial_{22}f(\phi(s))|\le_C \int_{I} |\partial_s(\partial_{22}(\rho_s-f)(\phi(s)))|ds + \int_{I}|\partial_{22}(\rho_s-f)(\phi(s))|ds\nonumber\\
   &\le_C \int_{I}|\partial_{222}(\rho_s-f)(\phi(s))\partial_s\phi(s)|ds + \int_{I}|\partial_{22}(\rho_s-f)(\phi(s))|ds\nonumber\\
   &\le_C\int_{0}^{1}|\partial_{222}(\rho_s-f)(x_1,x_2)|dx_2 + \int_{0}^{1}|\partial_{22}(\rho_s- f)(x_1,x_2)|dx_2\left(\sup_{s}|\partial_s\phi(s)|^{-1}\right),
    \end{align}
    where we used the change of variables $\phi(s)\to x_2$ in the last inequality. Note that
    \[
    |\partial_s\phi(s)| \ge |\partial_s\phi_0(s)|  - |\partial_s g(s)|\ge \gamma - \delta >C,
    \]
   where the second  inequality  is due to \eqref{rjjsdwdqdsd}. Hence, integrating \eqref{sympat} in $x_1$ over $\mathbb{T}$, we get
   \[
   \int_{\mathbb{T}} \sup_{s\in I} |\partial_{22}\rho_s(\phi(s)) - \partial_{22}f(\phi(s))|dx_1\le_C \rVert \rho_s-f\rVert_{W^{3,1}(\Omega)} + \rVert \rho_s-f\rVert_{W^{2,1}(\Omega)} \le_C \rVert \rho_s-f\rVert_{H^3(\Omega)}.
   \]
Combining this with \eqref{rhhsjdasd1sd}, we arrive at
    \begin{align}\label{gdidnot}
    \sup_{s\in I}\int_{\mathbb{T}}|\partial_{ss}g(x,s)|dx \le \rVert f-\overline{\rho}\rVert_{H^3(\Omega)}.
    \end{align}
Towards the proof,  we define $\phi_1$ and $h$ as follows:
  \[
  \phi_1 := \phi_0(s) + \frac{1}{2\pi}\int_{\mathbb{T}}g(x,s)dx,\quad h(x,s):=g(x,s)-\int_{\mathbb{T}}g(x,s)dx.
  \]
  Let us check if $\phi_1$ and $h$ satisfy the desired properties.   By its definition, we have  $\int_{\mathbb{T}}h(x,s)dx=0$ for each $s\in I$.  Also, \eqref{level_set_ff} and \eqref{phi_decomposition} tells us that
 \[
 f(x,\phi_1(s)+h(x,s)) = f(x,\phi_0(s)+g(x,s)) =f (x,\phi(x,s)) = s.
 \]
Collecting the estimates for $g$ obtained in \eqref{gsmall}, \eqref{rjjsdwdqdsd} and \eqref{gdidnot},  we see that \begin{align}\label{phi_1estm}
 \rVert h\rVert_{L^\infty}, \rVert \partial_s h\rVert_{L^{\infty}(\mathbb{T}\times I)},\  \rVert \partial_s(\phi_1-\phi_0)\rVert_{L^{\infty}(\mathbb{T}\times I)}, \  \rVert \partial_{ss}(\phi_1-\phi_0)\rVert_{L^{\infty}(\mathbb{T}\times I)}\le C\rVert f-\overline{\rho}\rVert_{H^3(\Omega)},
 \end{align}
 which proves \eqref{regularity_level}.
 To finish the proof, we have to prove  $f^*(\phi_1(s)) = s.$
To this end, observe that for each $s\in I$, it holds that
\[
\left|\left\{ (x_1,x_2)\in \Omega: \phi_1^{-1}> s\right\}\right| = \int_{\mathbb{T}}\int_{0}^11_{\left\{\phi_1^{-1}(x_2)\ge s\right\}}dx = \int_{\mathbb{T}}\int_0^{\phi_1(s)}dx = 2\pi\phi_1(s).
\]
We also have
\begin{align*}
\left|\left\{ (x_1,x_2)\in \Omega: {f}> s\right\}\right| &= \int_{\mathbb{T}}\int_{0}^1 1_{\left\{f(x_1,x_2)>s\right\}}dx =\int_{\mathbb{T}}\int_0^11_{\left\{x_2\le \phi(x_1,s)\right\}}dx=\int_\mathbb{T}\phi(x_1,s)dx_1 \\
&= 2\pi\left(\phi_0(s)+\frac{1}{2\pi}\int_{\mathbb{T}}g(x_1,s)dx_1\right)\\
& = 2\pi \phi_1(s).
\end{align*}
This implies that the areas of every super-level set of $\phi_1^{-1}$ and $f$ are equal. Recalling the definition of the rearrangement in \eqref{vertical_rearrangement}, we arrive at
\[
f^*(x_2) = \int_0^\infty 1_{0\le x_2\le |\left\{ f> s\right\}|} ds =\int_0^\infty 1_{0\le x_2\le |\left\{ \phi_1^{-1}> s\right\}|} ds = \phi^{-1}(x_2).
\]  This proves $f^*(\phi_1(s)) = s.$
  \end{proof}
  
  Recall that for $f\in L^1(\Omega)$, its potential energy is defined as
 \begin{align}\label{potential_1}
  E_P(f):=\int_\Omega f(x)x_2 dx.
  \end{align}
 It is well-known that for any stratified function $\rho_s$ is a critical point of $E_P$ under any divergence-free perturbation. In the case where $f$ is sufficiently close to $\rho_s$, the next proposition will quantitatively demonstrate that  $f^*$ is the unique minimizer of $E_P$ among all functions with the same area of super-level sets.  \begin{proposition}\label{propoos}
  Suppose $\rho_s$ satisfies \eqref{strati_sd2}. There exists $\delta=\delta(\gamma,\rVert \rho_s\rVert_{H^4})>0$ such that if $f={\rho}_s$ on $\partial\Omega$ and  $\rVert f-{\rho}_s\rVert_{H^3(\Omega)}\le \delta$, then
 \begin{align}\label{energy_nondegeneracy}
 C^{-1} \rVert f-f^*\rVert_{L^2(\Omega)}^2\le E_P(f)-E_P(f^*) \le C\rVert f-f^*\rVert_{L^2(\Omega)}^2.
 \end{align}
 Moreover, we have
 \begin{align}\label{hatedogs}
 \rVert \partial_1f \rVert_{L^2(\Omega)}\ge C\rVert f-f^*\rVert_{L^2(\Omega)}.
 \end{align}
 The constant $C$ depends  only on $\gamma$ and $\rVert \rho_s\rVert_{H^4}$.
  \end{proposition}
  \begin{proof}
 Let us prove \eqref{energy_nondegeneracy} first. For sufficiently small $\delta>0$,  Lemma~\ref{rearrangement_lem} ensures the existence of  $\phi_1,h$ such that 
  \begin{align}\
 \int_{\mathbb{T}}h(x,s)dx &= 0,\quad    f(x,\phi_1(s)+h(x,s))=s=f^*(\phi_1(s)),\label{level_32}
  \end{align}
  and
  \begin{align}\label{level_estimate32}
 \rVert\partial_{ss}(\phi_1-\phi_0) \rVert_{L^\infty}+  \rVert\partial_s\phi_1 -\partial_s\phi_0\rVert_{L^\infty} + \rVert h\rVert_{L^\infty}+\rVert \partial_sh\rVert_{L^\infty}\le C\delta,
  \end{align}
  where $\phi_0$ is the inverse of $\rho_s$, which trivially implies $\partial_s\phi_0(s)=\frac{1}{\partial_2\rho_s(\phi_0(s))}$. From \eqref{strati_sd2}, it follows that
  \begin{align}\label{lowerbs}
  0 < \frac{1}{C}\le -\partial_s\phi_0 <C.
  \end{align}
  Since $f^*$ is the inverse of $\phi_1$, the inverse function theorem and the estimates for $\phi_1$ in  \eqref{level_estimate32} tell us that
  \begin{align}\label{fs2t2ar1}
  \rVert \partial_{22}f^*-\partial_{22}\rho_s\rVert_{L^\infty(\Omega)}+  \rVert \partial_{2}f^*-\partial_2\rho_s\rVert_{L^\infty(\Omega)}\le  C\delta.
  \end{align}
  Denoting $\phi(x,s):=\phi_1(s)+h(x_1,s)$ and using  the change of variables, $x_2\to \phi(x_1,s)$, we have
  \begin{align}
   \rVert f-f^*\rVert_{L^2(\Omega)}^2 &= \int_{\mathbb{T}}\int_I |f(x,\phi(x,s))-f^*(\phi(x,s))|^2|\partial_s\phi(x,s)|dxds.\label{L2est}
  \end{align}
 Thanks to \eqref{level_32}, we have
 \begin{align}\label{fatagain}
 f(x,\phi(x,s))&-f^*(\phi(x,s))\nonumber\\
 & =s- f^{*}(\phi_1(s)) + \left(f^*(\phi_1(s)) -f^*(\phi_1(s) + h(x,s))\right)\nonumber\\
 & = f^*(\phi_1(s)) -f^*(\phi_1(s) + h(x,s))\nonumber\\
 & = \partial_2 f^*(\phi_1(s))h(x,s) + O \left(|\partial_{22}f^*||h(x_1,s)|^2\right)\nonumber\\
 & =\partial_2\rho_s h(x,s) + O\left(|\partial_2(\rho_s-f^*)||h(x_1,s)|\right) +O \left(|\partial_{22}f^*||h(x,s)|^2\right)\nonumber\\
 & = \partial_2\rho_sh(x,s) + O(\delta |h(x,s)|),
 \end{align}
 where  the last equality follows from \eqref{fs2t2ar1}. Note that \eqref{level_estimate32} also implies $|\partial_s\phi-\partial_s\phi_0|\le_C\delta$, thus \eqref{lowerbs} gives
   \begin{align}\label{changela}
  0 < \frac{1}{C}\le -\partial_s\phi \le C.
 \end{align}
 Plugging this and \eqref{fatagain} into  \eqref{L2est}, we obtain, for sufficiently small $\delta>0$,
 \begin{align}\label{hpt}
C^{-1}\rVert h\rVert_{L^2(\mathbb{T}\times [0,1])}\le \rVert f-f^*\rVert_{L^2(\Omega)}\le C\rVert h\rVert_{L^2(\mathbb{T}\times[0,1])}.
 \end{align}
 On the other hand, $E_P(f)$  can be written as
 \begin{align*}
 E_P(f)& = \int_{\mathbb{T}}\int_{0}^1 f(x_1,x_2)x_2 dx_2 \\
 &= -\int_{\mathbb{T}}\int_I f(x_1,\phi(x_1,s))\phi(x_1,s)\partial_{s}\phi(x_1,s)dsdx_1\\
 &= -\int_{\mathbb{T}}\int_I s \frac{1}{2}\partial_s\left(\phi(x_1,s) \right)^2dsdx_1.
 \end{align*}
 Since $f=f^*=\rho_s$ on $\partial \Omega$ and $\rho_s$ is strictly decreasing, we have 
 \[
 \phi(x_1,\rho_s(0))=\phi_1(\rho_s(0)) =0,\quad \phi(x_1,\rho_s(1)) = \phi_1(\rho_s(1)) = 1.
 \]
Using this, we can continue the computation above as
 \begin{align*}
 E_P(f)& = -\frac{1}{2}\int_{\mathbb{T}} \phi(x_1,\rho_s(0))^2 - \phi(x_1,\rho_s(1))^2dx_1 + \frac{1}{2}\int_{\mathbb{T}}\int_I (\phi_1(s)+h(x_1,s))^2dsdx_1\\
 &=\pi + \frac{1}{2}\int_{\mathbb{T}}\int_0^1 \phi_1(s)^2 +h(x_1,s)^2dsdx_1,
 \end{align*}
where the last equality follows from that $h$ has zero average in $x$. Similarly, we have
 \begin{align*}
 E_P(f^*) &= \int_{\mathbb{T}}\int_0^1 f^*(x_2)dx = - \int_{\mathbb{T}}\int_I f^*(\phi_1(s))\phi_1(s)\partial_s\phi_1(s)dsdx_1 =  \pi+ \frac{1}{2}\int_{\mathbb{T}}\int_I \phi_1(s)^2ds.
 \end{align*}
 Consequently, we arrive at $
 E(f) -E(f^*) = \frac{1}{2}\int_{\mathbb{T}}\int_I |h(x_1,s)|^2 dsdx_1.$
 Combining this with \eqref{hpt}, the estimates in \eqref{energy_nondegeneracy} is verified.
 
  Now, let us prove \eqref{hatedogs}.  Differentiating \eqref{level_32} in $x_1$, we see that
  \[
  0 = \partial_1 (f(x_1,\phi(x_1,s))) = \partial_1f(x_1,\phi(x,s))+\partial_2f(x_1,\phi(x_1,s))\partial_1h(x_1,s).
  \]
 Similarly, diffierentiating \eqref{level_32} in $s$ yields  $ 1=  \partial_2 f(x_1,\phi(x_1,s))\partial_s\phi(s)$, thus 
  \[
\partial_1 f(x_1,\phi(x_1,s)) =-\frac{\partial_1h(x_1,s)}{\partial_s\phi(x_1,s)}.
  \]
 Then, using the change of variables $x_2\to \phi(x_1,s)$ and also using \eqref{changela}, we obtain
 \[
 \rVert \partial_1 f\rVert_{L^2(\Omega)}\ge C\rVert \partial_1h\rVert_{L^2(\mathbb{T}\times I)}\ge C\rVert h\rVert_{L^2(\mathbb{T}\times I)},
 \] 
 where the last inequality follows from the zero-average in $x$ of $h$ in \eqref{level_32} and the Poincar\'e inequality. Therefore, combining this with \eqref{hpt}, we conclude that \eqref{hatedogs} holds.
  \end{proof}
  

\section{Stability in the IPM equation}\label{IPMsection}
In this section, we aim to prove the asymptotic stability for the incompressible porous media equation \eqref{IPM}. Throughout the section, we will fix
\begin{align}\label{tpypu}
\rho_s(x_2):=1-x_2.
\end{align}
\subsection{Preliminaries for the IPM}
Let us review important previous results concerning the local existence of the IPM equation \eqref{IPM} in the domain $\Omega=\mathbb{T}\times (0,1)$. For further details, we refer readers to the paper by Castro--C\'ordoba--Lear \cite{castro2019global}.

 We recall the following spaces from \cite[Section 1]{castro2019global}: For $k\in \mathbb{N}$, 
 \begin{align}\label{spacex}
 X^k(\Omega)&:=\left\{ f\in H^k(\Omega) :\partial_{2}^nf|_{\partial \Omega}=0,\text{ for }n=0,2,4,.\ldots k^\star\right\}, \text{ where }k^\star:=\begin{cases}
 k-2, &\text{ if $k$ is even},\\
 k-1 & \text{ if $k$ is odd}.
 \end{cases}
 \end{align}
 That is, $X^k(\Omega)$ is the closure of $\left\{ f\in C^\infty({\Omega}):\partial_{2}^nf|_{\partial \Omega}=0,\text{ for }n=0,2,4,.\ldots k^\star\right\}$ in the norm $H^k$. It is worth noting that the usual trace theorem, $H^k(\Omega)\hookrightarrow H^{k-1}(\partial\Omega)$, ensures that the vanishing normal derivatives in the above definition is well-defined.  For convenience, we  denote
 \[
 X^\infty(\Omega):=\cap_{k\in\mathbb{N}}X^k(\Omega).
 \] 
 By definition, it is clear from \eqref{cc_comv} that 
 \begin{align}\label{rra1}
 C^\infty_0(\Omega)\subset X^\infty(\Omega).
 \end{align}
   If $f\in X^\infty(\Omega)$, it holds that
  \begin{align}\label{tt1}
  \partial_1^{n_1}\partial_2^{2n_2}f\in X^\infty(\Omega),\text{ for any $n_1,n_2\in \mathbb{N}\cup\left\{0\right\}$, }
  \end{align}
  where $\partial_n$ denotes the normal derivative of $f$ on $\partial_\Omega$. Especially, we can use the integration by parts in the vertical variable without a boundary integral:
  \begin{align}\label{Xprop}
  \int_{\Omega} \partial_2^{k+1}f(x)\partial_2^{k+1}g(x)dx = - \int_{\Omega}\partial_2^{k}f(x)\partial_2^{k+2}g(x)dx,\text{ for any $k\in\mathbb{N}\cup\left\{0\right\}$ and $f,g\in X^\infty(\Omega)$.}
  \end{align}
In a usual domain without boundary, for instance $\mathbb{R}^2$, it is trivial that the Sobolev norms can be bounded by looking at only each single directional derivatives, that is, 
\[
\rVert f\rVert_{\dot{H}^k(\mathbb{R}^2)}\le C\left( \rVert \partial_1^{k}f\rVert_{L^2(\mathbb{R}^2)}+\rVert \partial_2^{k}f\rVert_{L^2(\mathbb{R}^2)}\right).
\] In a bounded domain, this property may depend on the boundary condition, since a mixed derivative might not be well controlled. While the next lemma seems intuitively obvious, we will give a proof for the sake of completeness, although the proof will be postponed to Appendix~\ref{Proof_oflemma3_1}.

 \begin{lemma}\label{onedirection}
 Let $f\in X^\infty(\Omega)$. For any $n,k\in \mathbb{N}\cup\left\{ 0\right\}$, we have
 \begin{align}\label{desired_es}
 \rVert \partial_{1}^n\partial_2^kf\rVert_{L^2(\Omega)}\le C_{n,k}\left(\rVert \partial_1^{n+k}f\rVert_{L^2(\Omega)} +\rVert \partial_2^{n+k}f\rVert_{L^2(\Omega)}  \right).
 \end{align}
 Consequently, we have
 \begin{align}\label{oned}
 \rVert f\rVert_{{H}^k}\le C_{k}\left( \rVert \partial_1^kf\rVert_{L^2} + \rVert \partial_2^k f\rVert_{L^2}\right) \text{ for all $k\ge 0$}.
 \end{align}
 \end{lemma}
 
  Let us consider a solution $\rho(t)$ to the IPM equation. We denote
 \begin{align}\label{def_rho2}
 \theta(t):=\rho(t)-\rho_s.
 \end{align}
  Substituting $\rho=\theta+{\rho}_s$ in \eqref{IPM}, one can easily see that $\theta(t)$ solves
 \begin{align}\label{IPMperturbed}
 \begin{cases}
 \theta_t + u\cdot\nabla \theta = u_2,\\
 u=\nabla^\perp \Psi,
 \end{cases}
 \text{ with }
 \begin{cases}
 -\Delta \Psi = \partial_1\theta, & \text{ in $\Omega$},\\
 \Psi = 0, & \text{ on  $\partial \Omega$.}
 \end{cases}
 \end{align}
The next lemma tells us that if the solution $\theta(t)\in X^k$, then the stream function $\Psi(t)$ behaves in a similar manner. 
    \begin{lemma}\cite[Lemma 3.1]{castro2019global}\label{lemma_psi}
    Let $f\in X^k(\Omega)$ and let $\Psi$ be a solution to
    \[
    \begin{cases}
    \Delta \Psi = - \partial_1f & \text{ in $\Omega$},\\
    \Psi = 0 & \text{ on $\partial \Omega$.}
    \end{cases}
    \]
    Then $\Psi\in X^{k+1}(\Omega)$ and it satisfies $\rVert \Psi\rVert_{H^{k+1}}\le C_k \rVert f\rVert_{H^k}$.
    \end{lemma}
 The local well-posedness to the equation~\eqref{IPMperturbed} was established in \cite{castro2019global}:
 \begin{theorem}\cite[Theorem 4.1]{castro2019global}\label{lwp}
 Let $k\in \mathbb{N}$ with $k\ge 3$. For any $\theta_0\in X^k(\Omega)$, there exists a time $T=T(\rVert \theta_0\rVert_{H^3})>0$ and a unique solution $\theta\in C(0,T; X^k(\Omega))$ for the equation \eqref{IPMperturbed}. 
 \end{theorem}
 
  Thanks to the local well-posedness theorem, we will  assume  that the initial data is smooth, that is,  $\theta_0\in X^\infty$. In view of the statement of Theorem~\ref{main_IPM}, the general case where $\theta_0\in H^k_0(\Omega)$ will be managed by  usual compactness argument in the end of the section.
  
\subsection{Energy estimates}\label{energy_conslea1}
In this section, we aim to derive an a priori energy estimate. The main result will be  given in Proposition~\ref{energy_IPM_estimate}. 

Let us recall the basic tame estimates concerning the Sobolev spaces. In the next lemma, we use the following notations: For $\alpha\in (\mathbb{N}\cup\left\{ 0\right\})^2$, 
 \[
 \alpha=(\alpha_1,\alpha_2),\quad |\alpha|:=\alpha_1+\alpha_2,\quad \partial^{\alpha}:=\partial_1^{\alpha_1}\partial_2^{\alpha_2}.
 \]
 \begin{lemma}\label{Simple_tame}\cite[Lemma 4.2]{castro2019global}
 Let $f,g\in C^\infty(\Omega)$. Then, for $\alpha,\beta\in (\mathbb{N}\cup\left\{ 0\right\})^2$, we have
 \begin{align*}
\rVert \partial^\alpha f\partial^\beta g\rVert_{L^2}&\le_{\alpha,\beta}  \rVert f\rVert_{H^{|\alpha|+|\beta|}}\rVert g\rVert_{L^\infty} +\rVert g\rVert_{H^{|\alpha|+|\beta|}}\rVert f\rVert_{L^\infty},\\
\rVert \partial^\alpha(fg)-f\partial^\alpha g\rVert_{L^2}&\le_{\alpha,\beta} \rVert f\rVert_{H^{|\alpha|}}\rVert g\rVert_{L^\infty}+\rVert f\rVert_{W^{1,\infty}}\rVert g\rVert_{H^{|\alpha|-1}}.
 \end{align*}
 \end{lemma}
  \begin{lemma}\label{average_difference}
For  $f\in H^1(\Omega)$ and  $\overline{f}(x_2):=\frac{1}{2\pi}\int_{\mathbb{T}} f(x_1,x_2)dx_1$, it holds that
\[
\rVert f-\overline{f}\rVert_{L^\infty}\le_C \rVert \partial_1f\rVert_{H^1}.
\]
\end{lemma}
\begin{proof}
 We notice the following pointwise estimate:
 \begin{align}\label{estimate8}
|(f-\overline{f})(x)|^2 = \left(\frac{1}{2\pi}\int_{\mathbb{T}}f(x_1,x_2)-f(z,x_2)dz\right)^2\le C\int_{\mathbb{T}}(f(x_1,x_2)-f(z,x_2))^2dz.
 \end{align}
The integrand in the right-hand side can be written as $
f(x_1,x_2)-f(z,x_2) = \int_{x_1}^{z}\partial_{1}f(a,x_2)da.$
For each fixed $x_1$, we apply the Sobolev embedding $H^1([0,1])\hookrightarrow L^\infty([0,1])$ to the map $x_2\mapsto \int_{x_1}^{z}\partial_{1}f(a,x_2)da$, yielding that 
 \[
 \sup_{x_2\in[0,1]}\left|\int_{x_1}^{z}\partial_1f(a,x_2)da\right|^2\le C\int_{0}^{1}\int_{x_1}^{z}|\partial_{12}f(a,y)|^2 + |\partial_{1}f(a,y)|^2dady\le \rVert \partial_1f\rVert_{H^1}^2.
 \] Therefore, taking the supremum over $x\in \Omega$ in \eqref{estimate8}, the desired result follows.
\end{proof}

\begin{proposition}\label{energy_IPM_estimate}
Let $\theta_0\in C^\infty_0(\Omega)$ and $\theta(t)\in C(0,T; X^\infty(\Omega))$ be the unique smooth solution to \eqref{IPMperturbed} for some $T>0$. For $k\ge 3$, it holds that
\begin{align*}
\frac{d}{dt}\left( \rVert \partial_1^k \theta\rVert_{L^2}^2 +  \rVert \partial_2^k \theta\rVert_{L^2}^2 \right)&\le -C_k(1-C_k\rVert \theta\rVert_{H^k})\rVert \nabla \Psi\rVert_{{H}^k}^2 + C_k\rVert u_2\rVert_{W^{1,\infty}}\rVert \theta\rVert_{H^k}^2.
\end{align*}
\end{proposition}
\begin{proof} In the proof, the implicit constant $C$ may depend on $k$, but its dependence will be omitted for simplicity.
In what follows  $\partial_i$ will denote either $\partial_1$ or $\partial_2$.
Using \eqref{IPMperturbed}, we compute
\begin{align}\label{energe1}
\frac{1}2\frac{d}{dt}\rVert \partial_i^k \theta\rVert_{L^2}^2 = -\int \partial_i^k(u\cdot\nabla \theta) \partial_i^k\theta dx +\int \partial_i^ku_2\partial_i^k\theta dx
\end{align}
We simplify the linear term first. Recalling from \eqref{IPMperturbed} that $u_2=\partial_1\Psi$ and $\partial_1\theta=-\Delta \Psi$, we have
\begin{align*}
\int \partial_i^ku_2\partial_i^k\theta dx&=\int \partial_1\partial_i^k\Psi \partial_i^k\theta dx = -\int \partial_i^k\Psi \partial_i^k\partial_1\theta dx = \int \partial_i^k\Psi \partial_i^k\Delta \Psi dx\nonumber\\
& = \int_{\partial\Omega}\partial_i^k\Psi \nabla(\partial_i^k \Psi)\cdot\vec{n}(x)d\sigma(x) - \int_{\Omega}|\nabla \partial_i^k\Psi|^2dx= -  \int_{\Omega}|\nabla \partial_i^k\Psi|^2dx,
\end{align*}
where the last equality follows from \eqref{Xprop}, which ensures that the integral over $\partial\Omega$  vanishes. Since $\theta\in X^\infty(\Omega)$, it follows from Lemma~\ref{lemma_psi} and Lemma~\ref{onedirection} that 
$\sum_{i=1,2}\int_{\Omega}|\nabla \partial_i^k\Psi|^2dx\ge_C \rVert\nabla \Psi\rVert_{\dot{H}^k}^2$. Since $\Psi=0$ on $\partial\Omega$,  the Poincar\'e inequality gives us $\rVert\nabla \Psi\rVert_{\dot{H}^k}^2\ge_C \rVert \nabla\Psi\rVert_{H^k}$, consequently, 
\begin{align}\label{dissipative_term}
\sum_{i=1,2}\int \partial_i^ku_2\partial_i^k\theta dx \le_C -\rVert\nabla \Psi\rVert_{{H}^k}^2.
\end{align}
Now, we move on to the nonlinear term. We write
\begin{align}\label{nerge2}
\int \partial_i^k(u\cdot\nabla \theta) \partial_i^k\theta dx &=\int \left(\partial_i^k(u\cdot\nabla \theta)-u\nabla \partial_i^k\theta\right) \partial_i^k\theta dx + \int u\cdot\nabla \partial_i^k \theta \partial_i^k\theta dx \nonumber\\
& = \int \left(\partial_i^k(u\cdot\nabla \theta)-u\nabla \partial_i^k\theta\right) \partial_i^k\theta dx + \int u\cdot \nabla \left(\frac{1}{2}(\partial_i^k\theta)^2\right)dx\nonumber\\
& = \int \left(\partial_i^k(u\cdot\nabla \theta)-u\nabla \partial_i^k\theta\right) \partial_i^k\theta dx,
\end{align}
where the last equality follows from the integration by parts and $u\cdot \vec{n}=0$ on the boundary. We claim that
\begin{align}\label{nontrivial_estimate}
\left|\int \left(\partial_i^k(u\cdot\nabla \theta)-u\nabla \partial_i^k\theta\right) \partial_i^k\theta dx\right|\le C\left( \rVert u_2\rVert_{W^{1,\infty}}\rVert \theta\rVert_{H^k}^2 +\rVert \nabla \Psi\rVert_{{H}^k}^2\rVert \theta\rVert_{H^k}  \right),
\end{align}
either $i=1$ or $i=2$. Once we have the above claim,  combining it  with \eqref{dissipative_term} yields the desired energy estimate. Thus, in the rest of the proof, we will aim to prove the claim \eqref{nontrivial_estimate}. We consider two cases, $i=1$ and $i=2$, separately.

\textbf{Case $i=1$.}
The Cauchy-Schwarz inequality gives us
\begin{align}\label{estimate_1}
\left|\int \left(\partial_1^k(u\cdot\nabla \theta)-u\nabla \partial_1^k\theta\right) \partial_1^k\theta dx\right|\le_C  \rVert \partial_1^k(u\cdot\nabla \theta)-u\nabla \partial_1^k\theta\rVert_{L^2}\rVert \partial_1^k\theta\rVert_{L^2},
\end{align}
while Lemma~\ref{Simple_tame} tells us that
\[
 \rVert \partial_1^k(u\cdot\nabla \theta)-u\nabla \partial_1^k\theta\rVert_{L^2}\le C\left(\rVert u\rVert_{W^{1,\infty}}\rVert \theta\rVert_{H^k}+\rVert u\rVert_{H^k}\rVert\nabla\theta\rVert_{L^\infty}\right).
\]
Since $k\ge 3$, the Sobolev inequality  gives us 
\begin{align}\label{estimate3}
\rVert u\rVert_{W^{1,\infty}}=\rVert\nabla\Psi\rVert_{W^{1,\infty}}\le C\rVert \nabla\Psi\rVert_{H^k},\text{ and }\rVert \nabla\theta\rVert_{L^\infty}\le C\rVert \theta\rVert_{H^k}.
\end{align}
Hence we have $ \rVert \partial_1^k(u\cdot\nabla \theta)-u\nabla \partial_1^k\theta\rVert_{L^2}\le C\rVert\nabla\Psi\rVert_{H^k}\rVert \theta\rVert_{H^k}$. Plugging this into \eqref{estimate_1}, we obtain
\[
\left|\int \left(\partial_1^k(u\cdot\nabla \theta)-u\nabla \partial_1^k\theta\right) \partial_1^k\theta dx\right|\le C \rVert\nabla\Psi\rVert_{H^k}\rVert\partial_1^k\theta\rVert_{L^2}\rVert \theta\rVert_{H^k}.
\]
Furthermore, from the Poisson equation in \eqref{IPMperturbed}, we find that
\[
\rVert \partial_1^k\theta\rVert_{L^2} = \rVert \partial_1^{k-1}\Delta\Psi\rVert_{L^2}\le C\rVert \nabla \Psi\rVert_{H^{k}},
\]
therefore, we conclude
\begin{align}\label{estimate_2}
\left|\int \left(\partial_1^k(u\cdot\nabla \theta)-u\nabla \partial_1^k\theta\right) \partial_1^k\theta dx\right|\le C \rVert\nabla\Psi\rVert_{H^k}^2\rVert \theta\rVert_{H^k},
\end{align}
which verifies \eqref{nontrivial_estimate}.

\textbf{Case $i=2$.}
Splitting $u\cdot\nabla=u_1\partial_1+u_2\partial_2$, we have
\begin{align}\label{I1I2}
 \int \left(\partial_2^k(u\cdot\nabla \theta)-u\nabla \partial_2^k\theta\right) \partial_2^k\theta dx&= \int \left(\partial_2^k(u_1\partial_1 \theta)-u_1 \partial_2^k\partial_1\theta\right) \partial_2^k\theta dx +  \int \left(\partial_2^k(u_2 \partial_2\theta)-u_2 \partial_2^k\partial_2\theta\right) \partial_2^k\theta dx\nonumber\\
 &=:I_1+I_2.
\end{align}
 The first integral $I_1$ can be estimated as before; applying the Cauchy-Schwarz inequality, we get
 \begin{align}\label{esimate4}
 |I_1|\le C\rVert \partial_2^k(u_1\partial_1\theta)-u_1\partial_2^k\partial_1\theta\rVert_{L^2}\rVert \partial_2^k\theta\rVert_{L^2},
 \end{align}
 while Lemma~\ref{Simple_tame} gives us 
  \begin{align*}
 \rVert \partial_2^k(u_1\partial_1\theta)-u_1\partial_2^k\partial_1\theta\rVert_{L^2}&\le C\left(\rVert u_1\rVert_{W^{1,\infty}}\rVert\partial_1\theta\rVert_{H^{k-1}} +\rVert u_1\rVert_{H^k}\rVert \partial_1\theta\rVert_{L^\infty}\right)\\
 &\le C\rVert \nabla \Psi\rVert_{H^k}\left(\rVert \partial_1\theta\rVert_{H^{k-1}}+\rVert \partial_1\theta\rVert_{L^\infty}\right),
  \end{align*}
  where we used the estimate for $u$ in \eqref{estimate3} to get the second inequality. Again, using the Sobolev inequality and the Poisson equation in \eqref{IPMperturbed}, we estimate
  \[
 \rVert \partial_1\theta\rVert_{L^\infty}\le_C \rVert \partial_1\theta\rVert_{H^{k-1}}=\rVert \Delta \Psi\rVert_{H^{k-1}} \le \rVert \nabla \Psi\rVert_{H^{k}},
  \]
which gives us $ \rVert \partial_2^k(u_1\partial_1\theta)-u_1\partial_2^k\partial_1\theta\rVert_{L^2}\le \rVert \nabla \Psi\rVert_{H^{k}}^2$.
Plugging this into \eqref{esimate4}, we conclude
\begin{align}\label{estimate5}
|I_1|\le C\rVert \nabla\Psi\rVert_{H^k}^2\rVert \partial_2^k\theta\rVert_{L^2}\le C\rVert \nabla\Psi\rVert_{H^k}^2\rVert \theta\rVert_{H^k}.
\end{align}

Now, let us estimate $I_2$ in \eqref{I1I2}. We denote
\begin{align}\label{thetabar}
\overline{\theta}(x_2):=\frac{1}{2\pi}\int_{\mathbb{T}}\theta(x_1,x_2)dx_1.
\end{align}
We split $I_2$ as 
\begin{align}\label{I22}
I_2 = \int \left(\partial_2^k(u_2 \partial_2(\theta-\overline{\theta}))-u_2 \partial_2^k\partial_2(\theta-\overline{\theta})\right) \partial_2^k\theta dx + \int \left(\partial_2^k(u_2 \partial_2\overline{\theta})-u_2 \partial_2^k\partial_2\overline{\theta}\right) \partial_2^k\theta dx=: I_{21} + I_{22}
\end{align}
 We estimate $I_{21}$ first. In a similar manner as above, the Cauchy-Schwarz inequality and the Sobolev inequalities yield
 \begin{align}\label{estimate9}
 |I_{21}|&\le_C\left(\rVert u_2\rVert_{W^{1,\infty}}\rVert\partial_2(\theta-\overline{\theta})\rVert_{H^{k-1}} +\rVert u_2\rVert_{H^k}\rVert \partial_2(\theta-\overline{\theta})\rVert_{L^\infty}\right)\rVert \theta\rVert_{H^k}\nonumber\\
 &\le_C\left(\rVert u_2\rVert_{W^{1,\infty}}\rVert(\theta-\overline{\theta})\rVert_{H^{k}} +\rVert u_2\rVert_{H^k}\rVert \partial_2(\theta-\overline{\theta})\rVert_{L^\infty}\right)\rVert \theta\rVert_{H^k}\nonumber\\
 &\le_C\rVert u_2\rVert_{W^{1,\infty}}\rVert \theta\rVert_{H^k}^2 + \rVert u_2\rVert_{H^k}\rVert \partial_2(\theta-\overline{\theta})\rVert_{L^\infty}\rVert \theta\rVert_{H^k}.
 \end{align}
 Moreover, Lemma~\ref{average_difference} implies
 \[
\rVert \partial_2(\theta-\overline{\theta})\rVert_{L^\infty}\le_C \rVert \partial_{12}\theta\rVert_{H^1}\le_C \rVert \partial_2\Delta\Psi\rVert_{L^2}\le_C \rVert \nabla\Psi\rVert_{H^k}, \]
 where the last inequality follows from $k\ge 3$. Plugging this and $\rVert u\rVert_{H^k}\le \rVert\nabla \Psi\rVert_{H^k}$  into \eqref{estimate9}, we obtain
 \begin{align}\label{estimate10}
 |I_{21}|\le C\left( \rVert u_2\rVert_{W^{1,\infty}}\rVert \theta\rVert_{H^k}^2  + \rVert\nabla \Psi\rVert_{H^k}^2\rVert\theta\rVert_{H^k}\right).
 \end{align}

 Next, we estimate $I_{22}$ in \eqref{I22}. By expanding $I_{22}$ using the product rule, we have
 \begin{align}\label{I222}
 I_{22}=\sum_{j=1}^{k}C_{k,j}\int \partial_{2}^ju_2\partial_2^{k-j+1}\overline{\theta}\partial_2^k\theta dx.
 \end{align}
 When $j=1$, we have
 \begin{align}\label{estimate11}
 \int \partial_{2}u_2\partial_2^{k}\overline{\theta}\partial_2^k\theta dx\le_C \rVert \partial_2u_2\rVert_{L^\infty}\rVert \overline{\theta}\rVert_{H^k}\rVert\theta\rVert_{H^k}\le_C \rVert \partial_2u_2\rVert_{L^\infty}\rVert \theta\rVert_{H^k}^2.
 \end{align}
  For $j\ge 2$, noting that $u_2=\partial_1\Psi$, we can apply the integration by parts in each integral as
 \begin{align*}
 \int \partial_{2}^ju_2\partial_2^{k-j+1}\overline{\theta}\partial_2^k\theta dx& = -\int \partial_2^j\Psi \partial_2^{k-j+1}\overline{\theta} dx\partial_1\partial_2^k\theta dx\\
 &=-\int_{\partial\Omega}\partial_2^j\Psi\partial_2^{k-j+1}\overline{\theta}\partial_1\partial_2^{k-1}\theta d\sigma(x) +\int_{\Omega}\partial_2\left(\partial_2^j\Psi\partial_2^{k-j+1}\overline{\theta} \right)\partial_1\partial_2^{k-1}\theta dx\\
 & = \int_{\Omega}\partial_2\left(\partial_2^j\Psi\partial_2^{k-j+1}\overline{\theta} \right)\partial_1\partial_2^{k-1}\theta dx.
 \end{align*}
 To see the last inequality, note that by Lemma~\ref{lemma_psi}, we have that $\Psi,\overline{\theta},\partial_1\theta$ are all in $X^\infty$. Since at least one of $j, k-j+1, k-1$ must be even, the definition of the space $X^k(\Omega)$ in \eqref{spacex} tells us that the boundary integral  must vanish.

 To continue,  we apply the Cauchy-Schwarz inequality to get 
 \begin{align}\label{estimate13}
\left| \int \partial_{2}^ju_2\partial_2^{k-j+1}\overline{\theta}\partial_2^k\theta dx\right| &=\left| \int_{\Omega}\partial_2\left(\partial_2^j\Psi\partial_2^{k-j+1}\overline{\theta} \right)\partial_1\partial_2^{k-1}\theta dx\right|\nonumber\\
&\le \left(\rVert \partial_{2}^{j+1}\Psi\partial_2^{k-j+1}\overline{\theta}\rVert_{L^2}+\rVert\partial_2^j\Psi\partial_2^{k-j+2}\overline{\theta}\rVert_{L^2}\right)\rVert\partial_1\partial_2^{k-1}\theta\rVert_{L^2}\nonumber\\
&\le \left(\rVert \partial_{2}^{j+1}\Psi\partial_2^{k-j+1}\overline{\theta}\rVert_{L^2}+\rVert\partial_2^j\Psi\partial_2^{k-j+2}\overline{\theta}\rVert_{L^2}\right) \rVert \nabla\Psi\rVert_{H^k}, \text{ for $j\ge 2$},
 \end{align}
 where we used $\rVert\partial_1\partial_2^{k-1}\theta\rVert_{L^2} =\rVert \partial_{2}^{k-1}\Delta\Psi\rVert_{L^2}\le\rVert\nabla\Psi\rVert_{H^k}$ to get the last inequality.

 Let us estimate $\rVert \partial_{2}^{j+1}\Psi\partial_2^{k-j+1}\overline{\theta}\rVert_{L^2}$ in \eqref{estimate13}. Since $2\le j\le k$, we have
 \begin{align}\label{estimate12}
 \rVert \partial_{2}^{j+1}\Psi\partial_2^{k-j+1}\overline{\theta}\rVert_{L^2}\le \rVert\nabla \Psi\rVert_{H^{k}}\rVert \partial_2^{k-j+1}\overline{\theta}\rVert_{L^\infty}\le  \rVert\nabla \Psi\rVert_{H^{k}(\Omega)}\rVert \overline{\theta}\rVert_{H^{k-j+2}}\le  \rVert\nabla \Psi\rVert_{H^{k}}\rVert \theta\rVert_{H^{k}},
 \end{align}
 where the second last inequality follows from the Sobolev inequality, noticing that  $\overline{\theta}$ depends only on the variable $x_2$.  
 
  Let us estimate $\rVert\partial_2^j\Psi\partial_2^{k-j+2}\overline{\theta}\rVert_{L^2}$ in \eqref{estimate13}. When $j=k$, we have
  \[
  \rVert\partial_2^j\Psi\partial_2^{k-j+2}\overline{\theta}\rVert_{L^2}\le_C \rVert \Psi\rVert_{H^k} \rVert\partial_{22}\overline{\theta}\rVert_{L^\infty}\le_C \rVert \nabla\Psi\rVert_{H^k}\rVert \theta\rVert_{H^k}.
  \]
 When $j\le k-1$, we have
  \[
  \rVert\partial_2^j\Psi\partial_2^{k-j+2}\overline{\theta}\rVert_{L^2}\le \rVert \partial_2^j\Psi\rVert_{L^\infty}\rVert \partial_2^{k-j+2}\theta\rVert_{L^2}\le \rVert \nabla \Psi\rVert_{H^k}\rVert \theta\rVert_{H^k}.
  \]
 where in the the last inequality, we used the Sobolev inequality and $j\ge 2$. Thus, we obtain 
 \[
 \rVert\partial_2^j\Psi\partial_2^{k-j+2}\overline{\theta}\rVert_{L^2}\le \rVert \nabla \Psi\rVert_{H^k}\rVert \theta\rVert_{H^k},\text{ for all $j\ge 2$.}
 \] Plugging this and \eqref{estimate12}  into \eqref{estimate13}, we get
 \[
 \left| \int \partial_{2}^ju_2\partial_2^{k-j+1}\overline{\theta}\partial_2^k\theta dx\right|\le C\rVert \nabla \Psi\rVert_{H^k}^2\rVert \theta\rVert_{H^k} \text{ for $j\ge 2$.}
 \]
Combining this with  \eqref{estimate11} and plugging them into \eqref{I222}, we see that
\[
|I_{22}|\le_C \rVert \partial_2u_2\rVert_{L^\infty}\rVert \theta\rVert_{H^k}^2 +\rVert \nabla \Psi\rVert_{H^k}^2\rVert \theta\rVert_{H^k}.
\]
Plugging this and \eqref{estimate10} into \eqref{I22}, we get
\[
|I_2|\le_C\rVert u_2\rVert_{W^{1,\infty}}\rVert \theta\rVert_{H^k}^2 +\rVert \nabla \Psi\rVert_{H^k}^2\rVert \theta\rVert_{H^k}.
\]
Plugging this and \eqref{estimate5} into \eqref{I1I2}, we conclude
\[
\left| \int \left(\partial_2^k(u\cdot\nabla \theta)-u\nabla \partial_2^k\theta\right) \partial_2^k\theta dx\right|\le_C \rVert u_2\rVert_{W^{1,\infty}}\rVert \theta\rVert_{H^k}^2 +\rVert \nabla \Psi\rVert_{H^k}^2\rVert \theta\rVert_{H^k}.
\]
Combining this with \eqref{estimate_2}, we obtain \eqref{nontrivial_estimate}.
\end{proof}

\subsection{Analysis of the energy structure}
The main objective in this subsection is to derive sufficient convergence rate of $\rho(t)$  to the equilibrium $\rho_0^*$, while $\rho$ stays close to ${\rho}_s$. We emphasize that $\theta(t)=\rho(t)-\rho_s$ will stay small but not necessarily decay. Thus it is important to distinguish the roles of   $\theta(t)$ and $\rho(t)-\rho_0^*$. We will consider the potential energy and the kinetic energy defined as
 \begin{align}\label{IPM_potential}
 E(t):= \int_{\Omega}(\rho(t)-\rho_0^*)x_2dx,\quad K(t):=\rVert u(t)\rVert_{L^2}^2.
 \end{align}
 Since we are interested in a solution that is close to $\rho_s$, we will assume, throughout this subsection, that
\begin{align}\label{assumption_delta}
 \rVert \theta(t)\rVert_{H^3}^2+  \int_0^T \rVert \nabla \Psi(t)\rVert_{H^k}^2 dt\le \delta\le \delta_0, \text{ for $t\in[0,T]$ for some $T>0$ and $\delta_0\ll 1$.}
 \end{align}
 
\begin{proposition}\label{main_IPM_decay}
Let $k\ge 3$. There exists $\delta_0=\delta_0(k)>0$ such that if $\rho(t)$ satisfies \eqref{assumption_delta}, then 
\begin{align}
\frac{d}{dt}E(t)= -K(t),\quad \frac{d}{dt}K(t)\le -C_k \rVert u_2\rVert_{L^2}^2. \label{officeteddy3}
 \end{align}
 \end{proposition}
 
\begin{proof}
Differentiating $E(t)$, we see that
 \begin{align}\label{potential_derivative}
 \frac{d}{dt}E(t) = \int_{\Omega}\rho_t x_2dx = -\int_{\Omega}u\cdot \nabla \rho x_2dx = \int_{\Omega}u_2\rho dx.
 \end{align}
 Using \eqref{IPM_stream}, we can further simplify the last expression as
 \[
  \int_{\Omega}u_2\rho dx = \int_{\Omega}\partial_1\Psi \rho dx = -\int_{\Omega}\Psi \partial_1\rho dx = \int_{\Omega}\Psi \Delta \Psi dx = -\int_{\Omega}|\nabla \Psi|^2 dx,
 \]
 therefore, we get 
 \begin{align}\label{energy_decay_IPM}
  \frac{d}{dt}E(t) =-\int_{\Omega}|\nabla \Psi|^2 dx = -K(t).
 \end{align}
In order to estimate $\frac{d}{dt}K(t)$, we differentiate \eqref{potential_derivative} and obtain
\[
\frac{1}{2}\left(\frac{d}{dt}\right)^2 E(t) =\frac{1}{2} \frac{d}{dt}\int_{\Omega} u_2\rho dx = \frac{1}{2}\int_{\Omega} u_2 \rho_t dx + \frac{1}{2}\int_{\Omega}\partial_tu_2 \rho dx.
\]
Using $u_2=\partial_1\Psi$ and $\partial_1\rho=\Delta\Psi$, we have
\[
\int_{\Omega}\partial_tu_2 \rho dx = -\int_{\Omega}\partial_t\Psi \partial_1\rho dx = \int_{\Omega}\partial_t\Psi \Delta \Psi dx= \int_{\Omega}\partial_t\Delta \Psi \Psi dx = -\int_{\Omega}{\partial_t}\partial_1\rho \Psi dx = \int_{\Omega}\partial_t\rho u_2 dx
\]
Therefore, we get
\begin{align*}
\frac{1}{2}\left(\frac{d}{dt}\right)^2 E(t) &= \int_{\Omega}u_2\rho_t = -\int_{\Omega} u_2 u \cdot \nabla \rho dx = -\int_{\Omega}u_2^2\partial_2\rho dx - \int_{\Omega}u_2u_1\partial_1\rho dx\\
&\ge -\int_{\Omega}u_2^2\partial_2\rho dx - \rVert u_2\rVert_{L^2}\rVert u_1\partial_1\rho\rVert_{L^2}.\end{align*}
Using the assumption on $\rVert \theta\rVert_{H^3}$ in \eqref{assumption_delta} and \eqref{tpypu}, we have 
\[
-\int_{\Omega}u_2^2\partial_2\rho dx = -\int u_2^2 \partial_2\rho_s dx -\int u_2^2 \partial_2\theta dx \ge \rVert u_2\rVert_{L^2}^2- \sqrt{\delta_0}\rVert u_2\rVert_{L^2}^2\ge C\rVert u_2\rVert_{L^2}^2 .
\]
Using \eqref{energy_decay_IPM}, we see that the above inequality implies
\begin{align}\label{rkawksmd}
C\frac{d}{dt}K(t) + \rVert u_2\rVert_{L^2}^2 \le C\rVert u_2\rVert_{L^2}\rVert u_1\partial_1\rho\rVert_{L^2}.
\end{align}
Let us estimate $\rVert u_1\partial_1\rho\rVert_{L^2(\Omega)}$. Using \eqref{IPM_stream}, we rewrite
\begin{align}\label{llap}
\rVert u_1\partial_1\rho\rVert_{L^2(\Omega)} = \rVert \partial_2\Psi\Delta\Psi\rVert_{L^2}\le \rVert \nabla\Psi\rVert_{L^4}\rVert \Delta\Psi\rVert_{L^4} \le \rVert \nabla\Psi\rVert_{H^3}\rVert \Psi\rVert_{L^2},
\end{align}
where the last inequality is due to the Gagliardo-Nirenberg interpolation theorem.
  ~Moreover, we notice from \eqref{IPM_stream} that  $g(x_2):=\int_{\mathbb{T}}\Psi(x_1,x_2)dx_1$ satisfies
\[
\partial_{22}g(x_2) = \int_{\mathbb{T}}\partial_{22}\Psi(x_1,x_2)dx_1 = \int_{\mathbb{T}}\Delta \Psi(x_1,x_2) - \partial_{11}\Psi (x_1,x_2)dx_1=\int_{\mathbb{T}}\partial_1(-\rho +\partial_1\Psi)dx=0, 
\]
with the boundary condition, $g(0)=g(1)=0$. Therefore, $g=0$ for all $x_2\in [0,1]$. In other words, the map $x_1\to \Psi(x_1,x_2)$ has zero average for each fixed $x_2$. Then, the Poincar\'e inequality tells us
\[
\rVert \Psi\rVert_{L^2}\le \rVert \partial_1\Psi\rVert_{L^2}=\rVert u_2\rVert_{L^2}.
\] 
Plugging this into \eqref{llap}, we get
\[
\rVert u_1\partial_1\rho\rVert_{L^2}\le \rVert \nabla\Psi\rVert_{H^3}\rVert u_2\rVert_{L^2}\le \rVert \theta\rVert_{H^3}\rVert u_2\rVert_{L^2}\le \sqrt{\delta_0} \rVert u_2\rVert_{L^2},
\]
where the second inequality follows from Lemma~\ref{lemma_psi} and the last inequality follows from \eqref{assumption_delta}. Plugging this into \eqref{rkawksmd}, we obtain that for sufficiently small $\delta_0>0$, 
$\frac{d}{dt}K(t)\le - C\rVert u_2\rVert_{L^2}^2.$
Combining this with \eqref{energy_decay_IPM}, we finish the proof of the proposition.
 \end{proof}

  \begin{corollary}\label{Cor_IPM}
  Let $k\ge 3$. There exists $\delta_0=\delta_0(k)>0$ such that if \eqref{assumption_delta} holds, then,
  \begin{align*}
 E(t)&\le  \frac{C\delta}{t^k},\\
  \frac{2}{t}\int_{t/2}^{t} K(s)ds&\le \frac{C\delta}{t^{k+1}}, \\
   \frac{2}{t}\int_{t/2}^{t}\rVert u_2(s)\rVert_{L^2}^2 ds &\le\frac{C\delta}{t^{k+2}},
  \end{align*}
  for all $t\in [0,T]$.
  \end{corollary}
 \begin{proof}
 Thanks to Proposition~\ref{propoos}, it holds that
 \begin{align}\label{compatibility_ipm}
 C^{-1}\rVert \rho(t)-\rho_0^*\rVert_{L^2}^2\le E(t)\le C\rVert \rho(t)-\rho_0^*\rVert_{L^2}^2 .
 \end{align}
  Now, using the Gagliardo-Nirenberg interpolation theorem, we observe that
 \[
 \rVert \partial_1\rho\rVert_{L^2(\Omega)} =\rVert \Delta\Psi\rVert_{L^2}\le C \rVert \nabla\Psi\rVert_{H^k}^{1/k}\rVert \nabla \Psi\rVert_{L^2}^{(k-1)/k}.
 \]
 On the other hand, applying \eqref{hatedogs}, we get $\rVert \partial_1\rho\rVert_{L^2}\ge C\rVert \rho-\rho_0^*\rVert_{L^2}$. Combining this with the above estimate, we find
 \[
 \rVert \nabla \Psi\rVert_{L^2}^2\ge C \left( \rVert \nabla \Psi\rVert_{H^k}^{-1/k}\rVert\rho-\rho_0^*\rVert_{L^2}\right)^{2k/(k-1)}\ge  C \rVert \nabla \Psi\rVert_{H^k}^{-2/(k-1)} E(t)^{k/(k-1)}.
 \]
 where the last inequality follows from \eqref{compatibility_ipm}. Hence, the variation of the potential energy in  \eqref{officeteddy3} must satisfy 
 \begin{align}\label{energyode}
 \frac{d}{dt}E(t)\le - C\left(\rVert \nabla \Psi\rVert_{H^k}^2\right)^{-1/(k-1)} E(t)^{k/(k-1)}.
 \end{align}
 Applying Lemma~\ref{ABC_ODE} with $\alpha=1/(k-1)>0$ and $n=k/(k-1)$ and $a(t)=\rVert \nabla\Psi\rVert_{H^k}^2$,  we get
 \begin{align}\label{chadl}
  E(t)&\le  \frac{C\delta}{t^k}.
 \end{align}
 Applying Lemma~\ref{ODEBDC}  to \eqref{officeteddy3} with $f(t)=E(t),\ g(t)=K(t),\ h(t)=C\rVert u_2(t)\rVert_{L^2}^2$, we get 
\[
\frac{2}{t}\int_{t/2}^{t} K(s)ds\le \frac{C}{t^{k+1}},\text{ and } \frac{2}{t}\int_{t/2}^{t}\rVert u_2(s)\rVert_{L^2}^2 ds \le\frac{C\delta}{t^{k+2}}.
\]
Together with \eqref{chadl}, we obtain the desired estimates.
 \end{proof}
 \subsection{Proof of Theorem~\ref{main_IPM}} Let $k\ge 3$ be fixed. Let $\delta_0$ be be fixed as  in Proposition~\ref{main_IPM_decay} . We claim that there exists $\epsilon_0(k)>0$ such that if $\theta_0:=\rho_0-{\rho}_s\in C^\infty_0(\Omega)$ and 
\begin{align*}
\rVert\rho_0-{\rho}_s\rVert_{H^k}\le \epsilon\le \epsilon_0,\end{align*}
 then for all $t>0$, it holds that
\begin{align}\label{exsproof1}
\rVert \theta(t)\rVert_{H^k}^2 +\int_0^t \rVert \nabla \Psi(t)\rVert_{H^k}^2 dt< C\epsilon^2,\text{ for some $C=C(k)>0$.}
\end{align}
Let us suppose for the moment that the claim is true. From Theorem~\ref{lwp}, we know that the maximal existence time depends only on $\rVert \theta(t)\rVert_{H^3}$, thus, the solution $\rho$ exists globally in time. Also, if necessary, we can further assume that $\epsilon$ is small enough to ensure Corollary~\ref{Cor_IPM} is applicable, that is, the solution $\rho(t)$ satisfies
\begin{align}\label{estimate_ipm14}
E(t)\le \frac{C\epsilon^2}{t^k},\quad \frac{2}{t}\int_{t/2}^{t} K(s)ds\le \frac{C\epsilon^2}{t^{k+1}}, \quad 
   \frac{2}{t}\int_{t/2}^{t}\rVert u_2(s)\rVert_{L^2}^2 ds \le\frac{C\epsilon^2}{t^{k+2}},\text{ for all $t>0$}.
\end{align}
In order to prove the theorem for $\rho_0\in H^k_0$ without assuming the smoothness, we can simply find an approximation $\rho_{0,n}$ such that $\rho_{0,n}-\overline{\rho}\in C^\infty_{0}(\Omega)$ such that $\rho_{0,n}\to \rho_0$ in $H^k$. Then  \eqref{estimate_ipm14} and \eqref{exsproof1} are satisfied by the global solutions $\rho_n(t)$, with initial data $\rho_{0,n}$. Since all estimates are uniform in $n$, the unique limit $\rho(t):=\lim_{n\to\infty}\rho_n(t)$  is the global in time solution satisfying the same estimates, that is,
\[
\rVert \rho(t)-\rho_s\rVert_{H^k}^2 \le \epsilon^2\quad \int_{\Omega}(\rho(t)-\rho_0^*)(x)x_2dx \le C\epsilon^2t^{-k},\text{ for all $t>0$}.
\]
Together with \eqref{energy_nondegeneracy}, we obtain all the desired estimates to establish Theorem~\ref{main_IPM}.

In the rest of the proof, we aim to prove \eqref{exsproof1}. Towards a contradiction, let $T^*>0$  be the first time that \eqref{exsproof1} breaks down, that is,
\begin{align}\label{reallyjoey}
\rVert \theta(T^*)\rVert_{H^k(\Omega)}^2 +\int_0^{T^*} \rVert \nabla \Psi(T^*)\rVert_{H^k(\Omega)}^2 dt= M\epsilon^2\ll 1,
\end{align} for some $M>0$, which will be chosen later. Towards a contradiction, let us denote
\[
f(t):= \rVert \partial_1^k \theta\rVert_{L^2}^2 +  \rVert \partial_2^k \theta\rVert_{L^2}^2.
\]
Since $\theta_0=\rho_0-\overline{\rho}\in C^\infty_0(\Omega)$, we have $\theta_0\in X^\infty(\Omega)$. Therefore it follows from Theorem~\ref{lwp} that $\theta(t)\in X^\infty(\Omega)$ for $t\in [0,T^*]$. Since $\theta(t)$ vanishes on the boundary,  Lemma~\ref{onedirection} tells us that
\begin{align}\label{estimate_ipm1}
C^{-1}\rVert \theta(t)\rVert_{H^k}^2\le f(t)\le C\rVert \theta(t)\rVert_{H^k}^2\le CM\epsilon^2,\quad f(0)\le C\epsilon^2.
\end{align}
Using the Sobolev embedding $W^{1,\infty}(\Omega)\hookrightarrow H^k(\Omega)$ for $k\ge3$, we have
\begin{align}\label{textbfs2}
\rVert u_2\rVert_{W^{1,\infty}}\le \rVert u_2\rVert_{H^k} \le \rVert\nabla \Psi\rVert_{H^k}\rVert \le \rVert \theta\rVert_{H^k}\le f(t)^{1/2}, 
\end{align}
where the second last inequality follows from \eqref{lemma_psi}.
Using this and Proposition~\ref{energy_IPM_estimate}, we observe that $f(t)$ satisfies
\begin{align}\label{rossgaller}
\frac{d}{dt}f(t)\le -C\rVert \nabla\Psi(t)\rVert_{H^k}^2 + Cf(t)^{3/2},\text{ for $t\in [0,T^*]$.}
\end{align}
 From this we see that $f$ satisfies $\frac{d}{dt}f(t)\le Cf^{3/2}$.  Since $f(0)\le C\epsilon^2$, one can easily find that
\begin{align}\label{lowerksd3sd}
f(t)\le \frac{C}{(\epsilon^{-1}-t)^2}, \text{ for $t\in [0,T^*]$.}
\end{align} 
This implies that for sufficiently small $\epsilon>0$, it holds that 
\begin{align}\label{lower_IPM}
f(t)\le C\epsilon^2,\text{ for $t\in [0,2]$.}
\end{align} In other words, for \eqref{reallyjoey} to happen, we must have 
\begin{align}\label{underboundT_IPM}
T^*> 2.
\end{align}
Now we analyze the energy inequality in a more careful manner. Using the Gagliardo-Nirenberg interpolation theorem, we see that 
\[
\rVert u_2\rVert_{W^{1,\infty}}\le \rVert u_2\rVert_{H^k}^{2/k}\rVert u_2\rVert_{L^2}^{1-2/k}\le \rVert \nabla\Psi\rVert_{H^k}^{2/k}\rVert u_2\rVert_{L^2}^{1-2/k}.
\]
Hence, using Young's inequality, we derive
\begin{align*}
\rVert u_2\rVert_{W^{1,\infty}}f(t)&\le \rVert \nabla\Psi\rVert_{H^k}^{2/k}\rVert u_2\rVert_{L^2}^{1-2/k}f(t) \\
& \le \eta\rVert \nabla\Psi\rVert_{H^k}^2 + C_{\eta}f(t)^{k/(k-1)}\rVert u_2\rVert_{L^2}^{(k-2)/(k-1)}\\
&\le \eta\rVert \nabla\Psi\rVert_{H^k}^2 + C_{\eta}(M\epsilon^2)^{k/(k-1)}\rVert u_2\rVert_{L^2}^{(k-2)/(k-1)}, \text{ for any $\eta>0$}.
\end{align*}
Substituting this into \eqref{rossgaller} with sufficiently small $\eta$ depending only on the implicit constant $C$, we arrive at
\[
f'(t)\le -C\rVert \nabla\Psi\rVert_{H^k}^2  +C(M\epsilon^2)^{k/(k-1)}\rVert u_2\rVert_{L^2}^{(k-2)/(k-1)},
\]Integrating the above in $t$ over $[1,T^*]$, we obtain
\begin{align}\label{ff1}
f(T^*)-f(1) + C_k\int_{1}^{T^*}\rVert\nabla \Psi(t)\rVert_{H^k}^2dt \le C_k(M\epsilon^2)^{k/(k-1)} \int_{1}^{T^*}\left(\rVert u_2(t)\rVert_{L^2}^2\right)^{\frac{k-2}{2(k-1)}}dt.
\end{align}
Thanks to  $M\epsilon^2\ll 1$ in \eqref{reallyjoey}, we apply Corollary~\ref{Cor_IPM} and obtain
\[
\frac{2}{t}\int_{t/2}^{t}\rVert u_2(s)\rVert_{L^2}^2 ds \le\frac{CM\epsilon^2}{t^{k+2}},\text{ for all $t\in [0,T^*]$.}
\]
Since $k\ge 3$, we have $n:=k+2>1$ and $\alpha:=\frac{k-2}{2(k-1)}\in (1/n,1)$. Thus applying Lemma~\ref{fialmass} with $f=\rVert u_2\rVert_{L^2}^2$ and $E=CM\epsilon^2$, we get
\[
\int_{1}^{T^*}\left(\rVert u_2(t)\rVert_{L^2}^2\right)^{\frac{k-2}{2(k-1)}}dt\le C(M\epsilon^2)^{\frac{k-2}{2(k-1)}}.
\]
Plugging this and \eqref{lower_IPM} into \eqref{ff1}, we obtain
\[
f(T^*) + C\int_{0}^{T^*}\rVert\nabla \Psi(t)\rVert_{H^k}^2dt\le C\left(\int_0^1\rVert\nabla \Psi\rVert_{H^k}^2dt +\epsilon^2 +(M\epsilon^2)^{\frac{3k-2}{2(k-1)}}\right)\le C \epsilon^2 + C(M\epsilon^2)^{\frac{3k-2}{2(k-1)}},
\]
where we used \eqref{textbfs2}  to justify the last inequality. Finally using $f(T^*)\ge C \rVert \theta\rVert_{H^k}^2$, which is due to the last inequality in \eqref{textbfs2}, we notice that for \eqref{reallyjoey} to hold, we must have\[
M\epsilon^2 =  \rVert \theta(T^*)\rVert_{H^k}^2 +\int_0^{T^*} \rVert \nabla \Psi(T^*)\rVert_{H^k(\Omega)}^2 dt\le_C f(T^*) + \int_{0}^{T^*}\rVert\nabla \Psi(t)\rVert_{H^k}^2dt \le C \epsilon^2 + C(M\epsilon^2)^{\frac{3k-2}{2(k-1)}}.
\]
Since $k\ge3$, we have $\frac{3k-2}{2(k-1)}>1$. Therefore our assumption $M\epsilon^2\ll1$ in \eqref{reallyjoey} and the above estimate give us
\[
M\epsilon^2\le C\epsilon^2.
\]
This leads to a contradiction with \eqref{reallyjoey}, if $M$ is chosen strictly larger than some implicit constant $C$, which depends only on $k$.  This finishes the proof of Theorem~\ref{main_IPM}.
\section{Stability in the Stokes transport system}\label{Stokessection}
In this section, our goal is to prove the asymptotic stability for the Stokes transport system \eqref{Stokes}. Throughout the section, we will assume that $\rho_s=\rho_s(x_2)$ is a function that is independent of $x_1$, and satisfies
\begin{align}\label{general_steady_stokes}
\gamma:=\inf_{x\in \Omega}\left(-\partial_2\rho_s(x_2)\right) >0,\quad \rVert \rho_s\rVert_{H^4(\Omega)} <\infty.
\end{align} 
In the rest of the paper, we allow the implicit constant $C$ to depend on $\gamma$ and $\rVert \rho_s\rVert_{H^4}$, but we omit its dependence in the notation for simplicity. While the proof will exhibit a structure quite similar to that of the previous section, we will furnish enough details for the sake of completeness of the paper.
\subsection{Preliminaries for the Stokes transport system}
Let us briefly review relevant results concerning the Stokes transport system in the periodic channel  For more details, we refer readers to the paper by Dalibard--Guillod--Leblond \cite{dalibard2023long} and additional references therein. 
 
  Let us consider a solution $\rho(t)$ to the Stokes transport system. As in the previous section we denote
 \begin{align}\label{def_rho252ss}
 \theta(t):=\rho(t)-\rho_s.
 \end{align}
  Substituting $\rho=\theta+{\rho}_s$ in \eqref{Stokes}, it immediately follows that
\begin{align}\label{Stokes_perturbed1}
\begin{cases}
\theta_t + u\cdot\nabla \theta = -\partial_2\rho_su_2,\\
\nabla \cdot u =0,\\
\end{cases}
with \ 
\begin{cases}
\Delta^2\Psi = \partial_1\theta, & \text{ in $\Omega$,}\\
\Psi=\nabla\Psi =0, & \text{ on $\partial\Omega$}.
\end{cases}
\end{align}
 Thanks to the Bilaplacian operator in the equation for the stream function, the  velocity in the Stokes transport system is much regular compared to that of the IPM equation. More quantitative estimate can be found in the next lemma.
    \begin{lemma}\cite[Lemma B.1]{dalibard2023long}\label{lemma_psi_stokes}
    Let $f\in H^k(\Omega)$ for $k\ge -2$ and $\Psi$ be a solution to
    \[
    \begin{cases}
    \Delta^2 \Psi =  f & \text{ in $\Omega$},\\
    \Psi =\nabla\Psi = 0 & \text{ on $\partial \Omega$.}
    \end{cases}
    \]
    Then  $\Psi\in H_{0}^2(\Omega)\cap H^{k+4}(\Omega)$ and $\rVert \Psi\rVert_{H^{k+4}}\le C_k \rVert f\rVert_{H^k}$.
    \end{lemma}
Thanks to more regular nature of the velocity, the global in time existence of the solution follows in a standard way. More precise statement can be formulated as below.
\begin{theorem}\cite[Lemma 2.1, Theorem A.1]{dalibard2023long}\label{lwp_stokes} Let $\theta_0\in H^k(\Omega)$ with $k\ge 3$. Then there exists a unique global solution $\rho(t)$ to \eqref{Stokes} such that $\rho\in C(\mathbb{R}^+; H^m(\Omega))$. Furthermore, denoting $\theta:=\rho-{\rho}_s$ and $\overline{\theta}(x_2):=\frac{1}{2\pi}\int_{\mathbb{T}}\theta(x_1,x_2)dx_1$, it holds that if $\theta_0=\partial_n\theta =\partial_n^2\overline{\theta}=0$ on $\partial\Omega$, then
\[
\theta(t)=\partial_2\theta(t)=\partial_2^2\overline{\theta}(t) = 0,\text{ on $\partial\Omega$.}
\]
\end{theorem} 

 In view of Theorem~\ref{main_Stokes}, we will consider initial data $\theta_0\in H^2_0(\Omega)\cap H^4(\Omega)$. The usual trace theorem ensures that  $\theta,\partial_2\theta$ and $\partial_{22}\overline{\theta}$ are indeed well-defined on the boundary $\partial\Omega$ and vanish identically as stated in the above theorem.
\subsection{Energy estimates}
In this subsection, our objective is to derive an a priori energy estimate. The main result is presented in Proposition~\ref{Stokes_perturb2ed1}. As one can notice from the proposition, we will focus on estimating the evolution of $\rVert \Delta^2\theta\rVert_{L^2}$, rather than directly examining $\rVert \theta\rVert_{H^4}$. The reason is to mitigate potential complications in the finer analysis of the anisotropic nature of the nonlinear term in the equation. As expected, the norm $\rVert\Delta^2\theta\rVert_{L^2}$ is equivalent to $\rVert\theta\rVert_{H^4}$ for solutions vanishing on the boundary. Although the immediate equivalence between these two norms may not be apparent through elementary integration by parts, Lemma~\ref{lemma_psi_stokes} provides a rigorous justification
    \begin{lemma}\label{bilap}
If $f\in H^4(\Omega)$ satisfies $f=\partial_2f =0$ on $\partial\Omega$ and $\int_{\Omega}f(x)dx=0$, then
\[
\rVert f\rVert_{{H}^2(\Omega)} \le C\rVert \Delta f\rVert_{L^2(\Omega)},\quad \rVert f\rVert_{H^4(\Omega)}\le C\rVert \Delta^2f\rVert_{L^2(\Omega)}.
\]
\end{lemma}
\begin{proof}
The only problematic part is to ensure that the mixed derivatives can be estimated by the Laplacian/Bilaplacian operator. The first inequality is trivial since $\rVert \Delta f\rVert_{L^2}^2 = \rVert \partial_{11}f\rVert_{L^2}^2 +\rVert \partial_{22}f\rVert_{L^2}^2+2\rVert\partial_{12}f\rVert_{L^2}^2$, which is strong enough to control all the second derivatives of $f$. For the second inequality,  we simply apply Lemma~\ref{lemma_psi_stokes} with $k=0$, yielding the desired result.
\end{proof}

\begin{proposition}\label{Stokes_perturb2ed1}
Let $\theta_0\in H^2_0(\Omega)\cap H^4(\Omega)$ and $\theta(t)$ be the unique smooth solution to \eqref{Stokes_perturbed1}. It holds that
\begin{align*}
\frac{d}{dt}\left(\rVert \Delta^2 \theta\rVert_{L^2}^2\right)&\le_C  -C(1- C\rVert \Delta^2 \theta\rVert_{L^2})\rVert \partial_1\Delta \theta\rVert_{L^2}^2 + \left(\rVert u_2\rVert_{H^3} +  \rVert u_2\rVert_{W^{2,\infty}}\right)\rVert \Delta^2\theta\rVert_{L^2}^2\\
& \  +   \rVert u_2\rVert_{H^3}\rVert \Delta^2\theta\rVert_{L^2}+\rVert \partial_1\theta\rVert_{L^2}^2.
\end{align*}
\end{proposition}
\begin{proof}
Using \eqref{Stokes_perturbed1}, we compute
\begin{align}\label{energe1_stokes}
\frac{1}2\frac{d}{dt}\rVert \Delta^2 \theta\rVert_{L^2}^2 = -\int \Delta^2 (u\cdot\nabla \theta)\Delta^2 \theta dx +\int \Delta^2 (-\partial_2\rho_s u_2)\Delta^2 \theta dx
\end{align}
We simplify the linear term first. Let us write
\[
\int \Delta^2 (-\partial_2\rho_s u_2)\Delta^2 \theta dx = \int \left(\Delta^2(-\partial_2\rho_s u_2)-(-\partial_2\rho_s)\Delta^2 u_2\right)\Delta^2 \theta dx + \int (-\partial_2\rho_s)\Delta^2 u_2\Delta^2 \theta dx=: A_1+A_2.
\] 
For $A_1$, the usual tame estimate yields
\begin{align}\label{mioko}
|A_1|&\le_C \rVert \Delta^2(-\partial_2\rho_s u_2)-(-\partial_2\rho_s)\Delta^2 u_2\rVert_{L^2}\rVert\Delta^2\theta\rVert_{L^2}\nonumber\\
&\le_C \left(\rVert \partial_2\rho_s\rVert_{L^\infty}\rVert u_2\rVert_{H^3}+ \rVert \partial_2\rho_s\rVert_{H^4}\rVert u_2\rVert_{L^\infty}\right)\rVert \Delta^2\theta\rVert_{L^2}\nonumber\\
&\le_C\rVert u_2\rVert_{H^3}\rVert \Delta^2\theta\rVert_{L^2}.
\end{align}
For $A_2$, we can write
\begin{align}\label{defiu1sdsxc2}
A_2 &= \int -\partial_2\rho_s \partial_1\Delta^2\Psi \Delta^2\theta dx = \int \partial_2\rho_s\partial_1\theta \partial_1\Delta^2\theta dx =  \int \Delta (\partial_2\rho_s \partial_1\theta) \partial_1\Delta\theta dx\nonumber\\
& = \int \partial_2\rho_s |\partial_1\Delta \theta|^2 dx + \int (\partial_{222}\rho_s\partial_1\theta + 2\partial_{22}\rho_s\partial_{12}\theta)\partial_1\Delta\theta dx=: A_{21}+A_{22},
\end{align}
while $A_{22}$ can be estimated as
\begin{align*}
|A_{22}|&\le_C \rVert \partial_{222}\rho_s\rVert_{L^\infty}\rVert \partial_1\theta\rVert_{H^1}\rVert \partial_1\Delta\theta\rVert_{L^2}\le C\rVert \partial_1\theta \rVert_{H^1}\rVert \partial_1\Delta\theta\rVert_{L^2}\\
&\le_C  \rVert \partial_1\theta\rVert_{L^2}^{1/2}\rVert \partial_1\theta\rVert_{H^2}^{3/2}\le \rVert \partial_1\theta\rVert_{L^2}^{1/2}\rVert \partial_1\Delta\theta\rVert_{L^2}^{3/2} \le \frac{\inf(-\partial_2\rho_s)}{4}A_{21} + C\rVert \partial_1\theta\rVert_{L^2}^2
 \end{align*}
where the last inequality follows from  Young's inequality. Thus, using  \eqref{general_steady_stokes}, we observe that \eqref{defiu1sdsxc2} can be written as 
\[
A_2 \le - \frac{3}{4}\inf(-\partial_2\rho_s)\rVert \partial_1\Delta\theta\rVert_{L^2}^2 + C\rVert \partial_1\theta\rVert_{L^2}^2\le -C\rVert \partial_1\Delta\theta\rVert_{L^2}^2+ C\rVert \partial_1\theta\rVert_{L^2}^2.
\]
Combining this with \eqref{mioko}, we get
\begin{align}\label{dissipative_term_Stokes}
\int \Delta^2 (-\partial_2\rho_s u_2)\Delta^2 \theta dx&\le -C\rVert\partial_1\Delta\theta\rVert_{L^2}^2+C\left( \rVert u_2\rVert_{H^3}\rVert \Delta^2\theta\rVert_{L^2}+\rVert \partial_1\theta\rVert_{L^2}^2\right).
\end{align}
where the last equality follows from Lemma~\ref{lwp_stokes}, which ensures that the integral over $\partial\Omega$ appearing in the integration by parts vanishes. 

  Now we estimate the integral coming from the nonlinear term in \eqref{energe1_stokes}. We derive a necessary estimate in a separate lemma. Indeed, Lemma~\ref{energy_lemma_viscous} gives us that
    \begin{align}\label{nonlinear_stokes_es}
  \left| \int \Delta^2(u\cdot\nabla \theta)\Delta^2\theta dx\right|\le_C  \left(\rVert u_2\rVert_{H^3} +  \rVert u_2\rVert_{W^{2,\infty}}\right)\rVert \Delta^2\theta\rVert_{L^2}^2 +\rVert u\rVert_{H^5}\rVert \partial_1\Delta \theta\rVert_{L^2}\rVert \Delta^2 \theta\rVert_{L^2}.
  \end{align}
 Using Lemma~\ref{lemma_psi_stokes},
 \[
 \rVert u\rVert_{H^5}\le_C \rVert \Psi \rVert_{H^5}\le_C \rVert \partial_1\theta\rVert_{H^1}\le_C \rVert \partial_1\Delta \theta\rVert_{L^2},
 \]  
 where the last inequality follows from Lemma~\ref{bilap}. Plugging this into \eqref{nonlinear_stokes_es}, and combining it with \eqref{dissipative_term_Stokes}, we conclude the proposition.
\end{proof}

\begin{lemma}\label{energy_lemma_viscous}
Let $u$ be a smooth divergence free vector field such that $u=0$ on $\partial \Omega$ and $f$ be a smooth scalar-valued function such that $f=\partial_2 f=\partial_{2}^2 \overline{f}=0$ on $\partial \Omega$ where $\overline{f}:=\frac{1}{2\pi}\int_{\mathbb{T}}f(x_1,x_2)dx_1$. Then, we have
\begin{align*}
&\left| \int_{\Omega}\Delta^2(u\cdot\nabla f)\Delta^2f dx\right|\le_C  \left(\rVert u_2\rVert_{H^3} +  \rVert u_2\rVert_{W^{2,\infty}}\right)\rVert \Delta^2f\rVert_{L^2}^2 +\rVert u\rVert_{H^5}\rVert \partial_1\Delta f\rVert_{L^2}\rVert \Delta^2 f\rVert_{L^2}
\end{align*}
\end{lemma}
\begin{proof}
We split the integral $\int_{\Omega} \Delta^2(u\cdot\nabla f)\Delta^2f dx$  as
 \begin{align*}
 \int_{\Omega} \Delta^2(u\cdot\nabla f)\Delta^2f dx &= \int_{\Omega} \left(\Delta^2(u\cdot\nabla f) - u\cdot\nabla\Delta^2f\right)\Delta^2\theta + \int_{\Omega}u\cdot\nabla \left(\frac{1}{2}(\Delta^2f)^2\right) dx\\
 & =\int_{\Omega} \left(\Delta^2(u\cdot\nabla f ) - u\cdot\nabla\Delta^2\theta\right)\Delta^2f dx,
 \end{align*}
 where the last integral vanishes due to the incompressibility of $u$ and the boundary condition $u=0$ on $\Omega$. Next, we further decompose the last integral as
 \begin{align}\label{I2es}
 \int_{\Omega} \Delta^2(u\cdot\nabla f)\Delta^2f dx & = \int_{\Omega} \left(\Delta^2(u\cdot\nabla f) - u\cdot\nabla\Delta^2f\right)\Delta^2fdx\nonumber\\
 &  = \int_{\Omega}\left(\Delta^2(u_1\partial_1f) - u_1\Delta^2\partial_1f\right)\Delta^2f dx +  \int_{\Omega}\left(\Delta^2(u_2\partial_2f) - u_2\Delta^2\partial_2f\right)\Delta^2f dx \nonumber\\
 &=: I_1 + I_2.
 \end{align}

 \textbf{Estimate for $I_1$.} Using the Cauchy-Schwarz inequality, we estimate
 \[
 |I_1|\le \rVert \Delta^2(u_1\partial_1f) - u_1\Delta^2\partial_1f\rVert_{L^2}\rVert \Delta^2f\rVert_{L^2}.
 \]
 Using Lemma~\ref{Simple_tame}, we have
 \begin{align*}
 \rVert \Delta^2(u_1\partial_1f) - u_1\Delta^2\partial_1f\rVert_{L^2}&\le \rVert u_1\rVert_{H^4}\rVert \partial_1f\rVert_{L^\infty} +\rVert u_1\rVert_{W^{1,\infty}}\rVert \partial_1f\rVert_{H^3}\\
 &\le \rVert u_1\rVert_{H^4}\rVert \partial_1f\rVert_{H^2} +\rVert u_1\rVert_{W^{1,\infty}}\rVert \partial_1f\rVert_{H^3}.
 \end{align*}
where the second inequality is due to the Sobolev embedding $H^2(\Omega)\hookrightarrow L^\infty(\Omega)$. Then, using Lemma~\ref{bilap}, we can replace $\rVert \partial_1f\rVert_{H^2} $ and $\rVert \partial_1f\rVert_{H^3}$ in the above estimate by $\rVert \partial_1\Delta f\rVert_{L^2}$ and $\rVert \Delta^2 f\rVert_{L^2}$, respectively.
 Hence, we obtain
 \begin{align}\label{I_1es}
 |I_1|\le \left( \rVert u_1\rVert_{H^4}\rVert \partial_1\Delta f\rVert_{L^2} +\rVert u_1\rVert_{W^{1,\infty}}\rVert \Delta^2 f\rVert_{L^2}\right)\rVert \Delta^2f\rVert_{L^2}.
\end{align}

 Before we start estimating  $I_2$, we consider the structure of the bi-Laplacian acting on a product of two functions. Expanding the bi-Laplacian using the formula $\Delta(fg)=\Delta g h + 2\nabla g\cdot\nabla h + g\Delta h$, we have 
 \begin{align}\label{Chandler_gone}
\Delta^2(gh)-f\Delta^2h &= \Delta (\Delta gh + 2\nabla g\cdot\nabla h +g\Delta h) - g\Delta^2 h\nonumber\\
& = \Delta^2 g h + 2\nabla \Delta g \cdot \nabla h + \Delta g\Delta h \nonumber\\
& \ + 2\left(\nabla \Delta g\cdot \nabla h+ \nabla\partial_1 g\cdot \nabla\partial_1h + \nabla\partial_2g\cdot\nabla\partial_2h + \nabla g\cdot\nabla\Delta h \right)\nonumber\\
& \ + \Delta g\Delta h + \nabla g\cdot\nabla\Delta h \nonumber\\
& =\Delta^2g h + 4\nabla\Delta g\cdot \nabla h + 2{\left(\nabla\partial_1 g\cdot \nabla\partial_1h + \nabla\partial_2g\cdot\nabla\partial_2h + \Delta g\cdot\Delta h \right)} + \nabla g\cdot \nabla\Delta h .
\end{align}

 \textbf{Estimate for $I_2$.} 
Next, we move on to estimate $I_2$ in \eqref{I2es}. Recalling our notation that $
\overline{f}(x_2):=\frac{1}{2\pi}\int_{\mathbb{T}}f(x_1,x_2)dx_2,$
we split $I_2$ as 
\begin{align}\label{I_2split}
I_2 &= \int_{\Omega}\left(\Delta^2(u_2\partial_2f) - u_2\Delta^2\partial_2f\right)\Delta^2f dx \nonumber\\
& = \int_{\Omega}\left(\Delta^2(u_2\partial_2(f-\overline{f})) - u_2\Delta^2\partial_2(f-\overline{f})\right)\Delta^2f dx + \int_{\Omega}\left(\Delta^2(u_2\partial_2\overline{f}) - u_2\Delta^2\partial_2\overline{f}\right)\Delta^2{f} dx\nonumber\\
& =: I_{21}+ I_{22}.
\end{align}
Again, the Cauchy-Schwarz inequality gives us
\[
|I_{21}|\le \rVert \Delta^2(u_2(f-\overline{f})) - u_2\Delta^2\partial_2(f-\overline{f})\rVert_{L^2}\rVert\Delta^2f\rVert_{L^2},
\]
while  Lemma~\ref{Simple_tame} yields
\begin{align*}
\rVert \Delta^2(u_2(f-\overline{f})) - u_2\Delta^2\partial_2(f-\overline{f})\rVert_{L^2}&\le_C \rVert u_2\rVert_{H^4}\rVert \partial_2(f-\overline{f})\rVert_{L^\infty} + \rVert  u_2\rVert_{W^{1,\infty}}\rVert f-\overline{f}\rVert_{H^4}\\
&\le_C \rVert u_2\rVert_{H^4}\rVert \partial_{12}f\rVert_{{H}^1} + \rVert  u_2\rVert_{W^{1,\infty}}\rVert f\rVert_{H^4}\\
&\le_C \rVert u_2\rVert_{H^4}\rVert \partial_{1}\Delta f\rVert_{L^2} + \rVert  u_2\rVert_{W^{1,\infty}}\rVert \Delta^2f\rVert_{L^2}.
\end{align*}
where the second inequality is due to  $\rVert \overline{f}\rVert_{H^4(\Omega)}\le \rVert f\rVert_{H^4(\Omega)}$ and Lemma~\ref{average_difference}, and the last inequality follows from  Lemma~\ref{bilap}. Hence, we get
\begin{align}\label{livinginnyc}
|I_{21}|\le \left(\rVert u_2\rVert_{H^4}\rVert \partial_{1}\Delta f\rVert_{L^2} + \rVert  u_2\rVert_{W^{1,\infty}}\rVert \Delta^2f\rVert_{L^2}\right)\rVert \Delta^2 f\rVert_{L^2}.
\end{align}

 Next, we estimate $I_{22}$.  Recalling $I_{22}$ from \eqref{I_2split} and applying \eqref{Chandler_gone} with $g=u_2, \ h=\partial_2\overline{f}$, we have
 \begin{align}\label{I_222}
 I_{22}&= \int_{\Omega} \Delta^2 u_2 \partial_2\overline{f}\Delta^2f dx \nonumber\\
 &+ \int_{\Omega}(4\nabla\Delta u_2\cdot\nabla\partial_2\overline{f} +2\left( \nabla \partial_1u_2\cdot\nabla\partial_{12}\overline{f} + \nabla\partial_2u_2\cdot\nabla\partial_{22}\overline{f}+\Delta u_2\Delta\partial_2\overline{f} + \nabla u_2\cdot\nabla\Delta\partial_2\overline{f}\right))\Delta^2f dx\nonumber\\
 & = \int_{\Omega} \Delta^2 u_2 \partial_2\overline{f}\Delta^2f dx +\int_{\Omega}(4\partial_2\Delta u_2\cdot\partial_{22}\overline{f} +2\left((\partial_{22}u_2+\Delta u_2)\cdot\partial_{222}\overline{f} + \partial_2 u_2\cdot\partial_{2222}\overline{f}\right))\Delta^2f dx\nonumber \\
 &=: I_{221}+ I_{222},
 \end{align}
 where the second last equality is due to the fact that $\overline{f}$ does not depend on the variable $x_1$.
 For $I_{222}$, using the Cauchy-Schwarz inequality, we just need a crude estimate,
 \begin{align}\label{I_2222}
 |I_{222}|&\le_C \left(\rVert u_2\rVert_{H^3(\Omega)}\rVert \partial_{22}\overline{f}\rVert_{L^\infty}+ \rVert u_2\rVert_{W^{2,\infty}}\rVert f\rVert_{H^4}\right)\rVert \Delta^2f\rVert_{L^2}\nonumber\\
 & \le_C \left(\rVert u_2\rVert_{H^3} +  \rVert u_2\rVert_{W^{2,\infty}}\right)\rVert f\rVert_{H^4}\rVert \Delta^2f\rVert_{L^2}\nonumber\\
 &\le_C  \left(\rVert u_2\rVert_{H^3} +  \rVert u_2\rVert_{W^{2,\infty}}\right)\rVert \Delta^2f\rVert_{L^2}^2,
  \end{align}
 where we used the Sobolev embedding $L^\infty([0,1])\hookrightarrow H^1([0,1])$ to have $\rVert \partial_{22}\overline{f}\rVert_{L^\infty}\le \rVert {f}\rVert_{H^4}$ in the the second inequality, and we used Lemma~\ref{bilap} in the last inequality.

 Lastly, we estimate $I_{221}$. Let $\Psi$ be the stream function of $u$ such that 
 \[
 \int_{\Omega}\Psi dx=0,\quad u=\nabla^{\perp}\Psi\text{ with }\Psi=\nabla\Psi=0 \text{ on $\partial\Omega$}.
 \] Indeed, such a stream function exists since $u$ is divergence free and, especially, $u=0$ on $\partial\Omega$. Using the integration by parts, we get
\begin{align*}
I_{221} & =\int_{\Omega}\partial_1 \Delta^2\Psi\partial_2\overline{f}\Delta^2 f dx = -\int \left(\Delta^2 \Psi \partial_2\overline{f}\right)\Delta^2\partial_1f dx =-\int \Delta\left(\Delta^2 \Psi \partial_2\overline{f}\right)\Delta\partial_1f dx,
\end{align*}
where the integration by parts in the last equality is justified by the assumption that $\partial_2\theta =\partial_{22}\overline{f}=0$ on $\partial \Omega$. Therefore, expanding $\Delta(\Delta^2\Psi \partial_2\overline{f})$, we get
\begin{align*}
I_{221} & = -\int (\Delta^3\Psi \partial_2\overline{f} + 2\partial_2\Delta^2\Psi \cdot\partial_{22}\overline{f} + \Delta^2\Psi \partial_{222}\overline{f})\partial_1\Delta\theta dx.
\end{align*}
Then,  we obtain
\begin{align*}
|I_{221}|&\le \rVert \Psi\rVert_{H^6}\rVert \overline{f}\rVert_{W^{3,\infty}}\rVert \partial_1\Delta f\rVert_{L^2} \le \rVert u\rVert_{H^5}\rVert \partial_1\Delta f\rVert_{L^2}\rVert \overline{f}\rVert_{H^4}\le \rVert u\rVert_{H^5}\rVert \partial_1\Delta f\rVert_{L^2}\rVert \Delta^2 f\rVert_{L^2},
\end{align*}
where we used Lemma~\ref{bilap} in the last inequality. Combining this with \eqref{I_2222} and \eqref{I_222}, we get
\[
|I_{22}|\le  \left(\rVert u_2\rVert_{H^3} +  \rVert u_2\rVert_{W^{2,\infty}}\right)\rVert \Delta^2f\rVert_{L^2}^2 +\rVert u\rVert_{H^5}\rVert \partial_1\Delta f\rVert_{L^2}\rVert \Delta^2 f\rVert_{L^2}.
\]
Combining this with \eqref{livinginnyc} and \eqref{I_2split}, we have
\[
|I_2|\le\left(\rVert u_2\rVert_{H^3} +  \rVert u_2\rVert_{W^{2,\infty}}\right)\rVert \Delta^2f\rVert_{L^2}^2 +\rVert u\rVert_{H^5}\rVert \partial_1\Delta f\rVert_{L^2}\rVert \Delta^2 f\rVert_{L^2}.
\]
Combining this with \eqref{I_1es} and \eqref{I2es}, we get
\[
\left|\int_{\Omega} \Delta^2(u\cdot\nabla \theta)\Delta^2\theta dx \right|\le \left(\rVert u_2\rVert_{H^3} +  \rVert u_2\rVert_{W^{2,\infty}}\right)\rVert \Delta^2f\rVert_{L^2}^2 +\rVert u\rVert_{H^5}\rVert \partial_1\Delta f\rVert_{L^2}\rVert \Delta^2 f\rVert_{L^2}.
\]
This finishes the proof.
\end{proof}

\subsection{Analysis of the energy structure}As in the previous section for the IPM equation, we define
 \begin{align}\label{Stokes_potential}
 E(t):= \int_{\Omega}(\rho(t)-\rho_0^*)x_2dx,\quad K(t):=\rVert \Delta \Psi(t)\rVert_{L^2}^2,
 \end{align}
 where $\rho_0^*$ is the vertical rearrangement of the initial density $\rho_0$. Throughout this subsection, we will assume that
 \begin{align}\label{assumption_delta_stokes}
 \rVert \theta(t)\rVert_{H^3}^2+  \int_0^T \rVert \Delta^2 \Psi(t)\rVert_{H^2}^2 dt\le \delta\le \delta_0\ll 1, \text{ for  $t\in[0,T]$ for some $T>0$ and $\delta_0>0$}.
 \end{align}
\begin{proposition}\label{main_Stokes_decay}
There exists $\delta_0=\delta_0(\gamma,\rVert \rho_s\rVert_{H^4})>0$ such that if \eqref{assumption_delta_stokes} holds, then 
 \begin{align}
\frac{d}{dt}E(t)= -K(t),\quad \frac{d}{dt}K(t)\le -C \rVert u_2\rVert_{L^2}^2. \label{officeteddy3_stokes}
 \end{align}
 \end{proposition}
 
\begin{proof}Differentiating the energy difference $E(t)$ in time, we see that
 \begin{align}\label{potential_derivative1}
 \frac{d}{dt}E(t) = \int_{\Omega}\rho_t x_2dx = -\int_{\Omega}u\cdot \nabla \rho x_2dx = \int_{\Omega}u_2\rho dx.
 \end{align}
 Using \eqref{Stokes_stream}, we can further simplify the last expression as
 \[
  \int_{\Omega}u_2\rho dx = \int_{\Omega}\partial_1\Psi \rho dx = -\int_{\Omega}\Psi \partial_1\rho dx =- \int_{\Omega}\Psi \Delta^2 \Psi dx = -\int_{\Omega}|\Delta \Psi|^2 dx,
 \]
 therefore, we get 
 \begin{align}\label{energy_decay_Stokes}
  \frac{d}{dt}E(t) =-\int_{\Omega}|\Delta \Psi|^2 dx = -K(t).
 \end{align}
In order to estimate $\frac{d}{dt}K(t)$, taking a derivative in \eqref{potential_derivative1}, we compute
\[
\frac{1}{2}\left(\frac{d}{dt}\right)^2 E(t) =\frac{1}{2} \frac{d}{dt}\int_{\Omega} u_2\rho dx = \frac{1}{2}\int_{\Omega} u_2 \rho_t dx + \frac{1}{2}\int_{\Omega}\partial_tu_2 \rho dx.
\]
Using $u_2=\partial_1\Psi$ and $\partial_1\rho=\Delta^2\Psi$, we have
\[
\int_{\Omega}\partial_tu_2 \rho dx = -\int_{\Omega}\partial_t\Psi \partial_1\rho dx =- \int_{\Omega}\partial_t\Psi \Delta^2 \Psi dx=- \int_{\Omega}\partial_t\Delta^2 \Psi \Psi dx = -\int_{\Omega}{\partial_t}\partial_1\rho \Psi dx = \int_{\Omega}\partial_t\rho u_2 dx
\]
Therefore, we get
\begin{align*}
\frac{1}{2}\left(\frac{d}{dt}\right)^2 E(t) &= \int_{\Omega}u_2\rho_t = -\int_{\Omega} u_2 u \cdot \nabla \rho dx = -\int_{\Omega}u_2^2\partial_2\rho dx - \int_{\Omega}u_2u_1\partial_1\rho dx\\
&\ge -\int_{\Omega}u_2^2\partial_2\rho dx - \rVert u_2\rVert_{L^2}\rVert u_1\partial_1\rho\rVert_{L^2}.
\end{align*}
Using the assumption on $\rVert \theta\rVert_{H^3}$ in \eqref{assumption_delta_stokes}, we have 
\[
-\int_{\Omega}u_2^2\partial_2\rho dx = -\int u_2^2 \partial_2\rho_s dx - \int u_2^2 \partial_2(\rho-\rho_s)dx\ge \gamma \rVert u_2\rVert_{L^2}^2 -\sqrt{\delta_0}\rVert u_2\rVert_{L^2}^2\ge  C\rVert u_2\rVert_{L^2}^2,
\]
where we used the strict positivity of $\gamma$ which is assumed in \eqref{general_steady_stokes}.
Using \eqref{energy_decay_Stokes}, we see that the above inequality implies
\begin{align}\label{rkawksmd_stokes}
\frac{d}{dt}K(t) + C\rVert u_2\rVert_{L^2}^2 \le C\rVert u_2\rVert_{L^2}\rVert u_1\partial_1\rho\rVert_{L^2}.
\end{align}
Let us estimate $\rVert u_1\partial_1\rho\rVert_{L^2}^2$. Using \eqref{Stokes_stream}, we rewrite
\begin{align}\label{llap_stokes}
\rVert u_1\partial_1\rho\rVert_{L^2(\Omega)} = \rVert \partial_2\Psi\Delta^2\Psi\rVert_{L^2}\le \rVert \nabla\Psi\rVert_{L^4}\rVert \Delta^2\Psi\rVert_{L^4} \le \rVert \Delta^2\Psi\rVert_{H^2}\rVert \Psi\rVert_{L^2},
\end{align}
where the last inequality is due to the Gagliardo-Nirenberg interpolation theorem. Moreover, we notice from \eqref{Stokes_stream} that  $g(x_2):=\int_{\mathbb{T}}\Psi(x_1,x_2)dx_1$ satisfies
\begin{align*}
\partial_{2222}g(x_2) &= \int_{\mathbb{T}}\partial_{2222}\Psi(x_1,x_2)dx_1 = \int_{\mathbb{T}}\Delta^2 \Psi(x_1,x_2) - (\partial_{1111}+2\partial_{1122})\Psi (x_1,x_2)dx_1=\int_{\mathbb{T}}\partial_1\rho dx=0, 
\end{align*}
with the boundary condition, $g=\partial_2 g=0$ on $\partial\Omega$. Therefore, $g=0$ for all $x_2\in [0,1]$. In other words, the map $x_1\to \Psi(x_1,x_2)$ has zero average for each fixed $x_2$. Then, the Poincar\'e inequality tells us
\[
\rVert \Psi\rVert_{L^2}\le \rVert \partial_1\Psi\rVert_{L^2}=\rVert u_2\rVert_{L^2}.
\] 
Substituting this into \eqref{llap_stokes}, we get
\[
\rVert u_1\partial_1\rho\rVert_{L^2}\le \rVert \Delta^2\Psi\rVert_{H^2}\rVert u_2\rVert_{L^2}\le \rVert \rho-{\rho}_s\rVert_{H^3}\rVert u_2\rVert_{L^2}\le \sqrt{\delta_0} \rVert u_2\rVert_{L^2},
\]
where the second inequality follows from Lemma~\ref{lemma_psi_stokes} (noting that $\Delta^2\Psi = \partial_1(\rho-\rho_s)$ and $\rho-\rho_s=0$ on $\partial\Omega$) and the last inequality follows from \eqref{assumption_delta_stokes}. Plugging this into \eqref{rkawksmd_stokes}, we obtain that for sufficiently small $\delta>0$, 
$\frac{d}{dt}K(t)\le - C\rVert u_2\rVert_{L^2}^2.$
Combining this with \eqref{energy_decay_Stokes}, we finish the proof of the proposition.
 \end{proof}
 
  \begin{corollary}\label{Cor_Stokes}
  There exists $\delta_0=\delta_0(\gamma,\rVert \rho_s\rVert_{H^4})>0$ such that if \eqref{assumption_delta_stokes} holds, then,
  \begin{align*}
 E(t)&\le  \frac{C\delta}{t^2},\\
  \frac{2}{t}\int_{t/2}^{t} K(s)ds&\le \frac{C\delta}{t^{3}}, \\
   \frac{2}{t}\int_{t/2}^{t}\rVert u_2(s)\rVert_{L^2}^2 ds &\le\frac{C\delta}{t^{4}},
  \end{align*}
  for all $t\in [0,T]$.
  \end{corollary}
 \begin{proof}
 Thanks to Proposition~\ref{propoos}, we know that there is a constant $C>0$ such that
 \begin{align}\label{compatibility_Stokes}
 C^{-1}\rVert \rho(t)-\rho_0^*\rVert_{L^2}^2\le E(t)\le C\rVert \rho(t)-\rho_0^*\rVert_{L^2}^2 .
 \end{align}
  Now, using the Gagliardo-Nirenberg interpolation theorem, we estimate
 \[
 \rVert \partial_1\rho\rVert_{L^2(\Omega)} =\rVert \Delta^2\Psi\rVert_{L^2}\le C \rVert \Delta^2\Psi\rVert_{H^2}^{1/2}\rVert \Delta \Psi\rVert_{L^2(\Omega)}^{1/2}.
 \]
 On the other hand, applying \eqref{hatedogs}, we observe that $\rVert \partial_1\rho\rVert_{L^2}\ge C\rVert \rho-\rho_0^*\rVert_{L^2}$. Combining this with the above estimate, we get
 \[
 \rVert \Delta \Psi\rVert_{L^2}^2\ge \rVert \Delta^2\Psi\rVert_{H^2}^{-2}\rVert \rho-\rho_0^*\rVert_{L^2}^4\ge  C \rVert \Delta^2 \Psi\rVert_{H^2}^{-2} E(t)^2.
 \]
 where the last inequality follows from \eqref{compatibility_Stokes}. Hence, the variation of the potential energy in  \eqref{officeteddy3_stokes} yields 
 \begin{align}\label{energyode_Stokes}
 \frac{d}{dt}E(t)\le - C\left(\rVert \Delta^2 \Psi\rVert_{H^2}^2\right)^{-1} E(t)^{2}.
 \end{align}
 Applying Lemma~\ref{ABC_ODE} with $\alpha=1>0$ and $n=2$ and $a(t)=\rVert \Delta^2\Psi\rVert_{H^2}^2$,  we get
 \[
 E(t)\le \frac{CA}{t^{2}}, \text{ where $A=\int_0^t \rVert \Delta^2 \Psi\rVert_{H^2}^{2}ds$.}
 \]
Under the assumptions in \eqref{assumption_delta}, we have $A\le C\delta$. Therefore,
 \begin{align}\label{chadl_stokes}
  E(t)\le \frac{C\delta}{t^{2}}.
 \end{align}
 Applying Lemma~\ref{ODEBDC} with $f(t)=E(t),\ g(t)=K(t),\ h(t)=C\rVert u_2(t)\rVert_{L^2}^2$, it follows from \eqref{officeteddy3_stokes} that
\[
\frac{2}{t}\int_{t/2}^{t} K(s)ds\le \frac{C}{t^{3}},\text{ and } \frac{2}{t}\int_{t/2}^{t}\rVert u_2(s)\rVert_{L^2}^2 ds \le\frac{C\delta}{t^{4}}.
\]
Together with \eqref{chadl_stokes}, we obtain the desired estimates.
 \end{proof}

\subsection{Proof of Theorem~\ref{main_Stokes}}  Let $\delta_0$ be  fixed so that  Proposition~\ref{main_Stokes_decay} and Corollary~\ref{Cor_Stokes} are applicable. We claim that there exists $\epsilon=\epsilon(\gamma,\rVert \rVert \rho_s\rVert_{H^4})>0$ such that if $\rho_0-{\rho}_s\in H^2_0(\Omega)\cap H^4(\Omega)$ and 
\begin{align}\label{initksd2sdsd}
\rVert\rho_0-{\rho}_s\rVert_{H^4}\le \epsilon\ll 1,\end{align}
 then for all $t>0$, it holds that
\begin{align}\label{exsproof1_Stokes}
\rVert \Delta^2 \theta(t)\rVert_{L^2(\Omega)}^2 +\int_0^t \rVert \Delta^2 \Psi(t)\rVert_{H^2(\Omega)}^2 dt< C\epsilon^2,\text{ for some $C>0$.}
\end{align}
Let us suppose for the moment that the claim is true. Then by Corollary~\ref{Cor_Stokes},  the solution $\rho(t)$ satisfies
\begin{align}\label{estimate_ipm12214}
E(t)\le \frac{C\epsilon^2}{t^2},\quad \frac{2}{t}\int_{t/2}^{t} K(s)ds\le \frac{C\epsilon^2}{t^{3}}, \quad 
   \frac{2}{t}\int_{t/2}^{t}\rVert u_2(s)\rVert_{L^2}^2 ds \le\frac{C\epsilon^2}{t^{4}},\text{ for all $t>0$}.
\end{align}
Especially, the potential energy estimate in \eqref{estimate_ipm12214} and \eqref{energy_nondegeneracy} give us all the necessary properties to establish Theorem~\ref{main_Stokes}. 

In the rest of the proof, we aim to prove \eqref{exsproof1_Stokes}. Towards a contradiction, let $T^*>0$  be the first time that \eqref{exsproof1_Stokes} breaks down, that is,
\begin{align}\label{reallyjoey_Stokes}
\rVert \Delta^2 \theta(T^*)\rVert_{L^2(\Omega)}^2 +\int_0^{T^*} \rVert \Delta^2 \Psi(t)\rVert_{H^2(\Omega)}^2 dt= M\epsilon^2\ll1 ,
\end{align} for some $M\gg 1$, which will be chosen later. Towards a contradiction, let us denote
\[
f(t):= \rVert \Delta^2\theta(t) \rVert_{L^2}^2,\quad g(t):=\rVert \Delta^2\Psi(t)\rVert_{H^2}^2.\]
Since $\theta_0:=\rho_0-\overline{\rho}\in H^2_0\cap H^4$,  it follows from Theorem~\ref{lwp_stokes} that $\theta(t)$ satisfies $\theta=\partial_2\theta=\partial_{22}\overline{\theta}=0$ on $\partial\Omega$. Then, Lemma~\ref{bilap} and our assumption on $T^*
$  tell us that
\begin{align}\label{estimate_stokes}
C^{-1}g(t) \le {C}^{-1}\rVert \theta(t)\rVert_{H^4}^2\le f(t)\le C\rVert \theta(t)\rVert_{H^4}^2\le CM\epsilon^2, \text{ for all $t\in[0,T^*]$},\quad f(0)\le C\epsilon^2,
\end{align}
where the initial condition is due to  \eqref{initksd2sdsd}. In particular, we have
\[
\rVert \partial_1\Delta \theta\rVert_{L^2}^2=\rVert \Delta \partial_1\theta\rVert_{L^2}^2 \ge \rVert \partial_1\theta\rVert_{H^2}^2 = \rVert \Delta^2\Psi\rVert_{H^2}^2 = g(t).
\] Using this and \eqref{estimate_stokes}, we derive from Proposition~\ref{Stokes_perturb2ed1} that
\begin{align}\label{rossgaller_stokes}
\frac{d}{dt}f(t)\le -Cg(t) + C\left(\rVert u_2\rVert_{H^3} +  \rVert u_2\rVert_{W^{2,\infty}}\right)f(t) +   C \rVert u_2\rVert_{H^3}\rVert \Delta^2\theta\rVert_{L^2}+\rVert \partial_1\theta\rVert_{L^2}^2,\text{ for $t\in [0,T^*]$.}
\end{align}
With this differential inequality, we first derive a crude estimate for "short time" first and more careful analysis will follows afterwards.
 
Using the Sobolev embedding $W^{2,\infty}(\Omega)\hookrightarrow H^4(\Omega)$, we have
\begin{align}\label{rtjsdsdsd1sd}
\rVert u_2\rVert_{H^3} +  \rVert u_2\rVert_{W^{2,\infty}}\le  \rVert u_2\rVert_{H^4} \le \rVert\nabla \Psi\rVert_{H^4} \le \rVert \partial_1\theta\rVert_{H^1}\le \sqrt{f(t)}, 
\end{align}
where the  last inequality follows from \eqref{estimate_stokes}. From this and \eqref{rossgaller_stokes} we see that $f$ satisfies
\[
\frac{d}{dt}f(t) \le C(f(t)^{3/2} + f(t))\le C f(t),\quad \text{ for $t\in[0,T^*]$}.
\]
Using $f(0)\le C\epsilon^2$, we can immediately deduce from this differential inequality that $f(t)\le C\epsilon^2e^{Ct},$ for $t\in [0,T^*]$. Since $g(t)\le Cf(t)$,  we can easily deduce that
\begin{align}\label{rjjsdpwdwd1sd}
f(t) +  \int_0^{t}g(s)ds\le C\epsilon^2 e^{Ct}\text{ for $t\in [0,T^*]$}.
\end{align}
Thus, for \eqref{reallyjoey_Stokes} to occur, we must have $M\epsilon^2 \le_C \epsilon^2 e^{CT^*}$. This gives us a lower bound of $T^*$,
\[
T^*\ge C \log M.
\]
Let us pick 
\begin{align}\label{lmiddlet}
T^*_M:=\log\log M\gg 1.
\end{align} Without loss of generality, we can assume $M$ is sufficiently large to ensure that 
\[
T^*_M < T^*.
\]
Then , it follows from \eqref{rjjsdpwdwd1sd} that\begin{align}\label{middleestimagte}
f(T^*_M) + \int_0^{T^*_M} g(t)dt \le C\epsilon^2 (\log M)^C.
\end{align}
 
  Now, we will estimate $f(t)$ more carefully for $t \ge T^*_M$. Using the Gagliardo-Nirenberg interpolation theorem, we see that 
\begin{align}\label{brovksd1sd}
\rVert u_2\rVert_{W^{2,\infty}} +\rVert u_2\rVert_{H^3}\le_C \rVert u_2\rVert_{L^2}^{2/5}\rVert u_2\rVert_{H^5}^{3/5}\le_C \rVert u_2\rVert_{L^2}^{2/5}\rVert\Delta^2 \Psi\rVert_{H^2}^{3/5}.
\end{align}
Hence, Young's inequality can give us
\begin{align}\label{spring1}
\left(\rVert u_2\rVert_{W^{2,\infty}} +\rVert u_2\rVert_{H^3} \right)f(t)&\le_C  \rVert u_2\rVert_{L^2}^{2/5}\rVert\Delta^2 \Psi\rVert_{H^2}^{3/5}f(t) \nonumber\\
& \le_C \eta\rVert \Delta^2\Psi\rVert_{H^2}^2 + C_\eta f(t)^{10/7}\rVert u_2\rVert_{L^2}^{4/7}\nonumber\\
&\le_C \eta g(t) + C_\eta (M\epsilon^2)^{10/7}\rVert u_2\rVert_{L^2}^{4/7}, \text{ for any $\eta>0$}.
\end{align}
Again using \eqref{brovksd1sd} and the estimate  $\rVert \theta(t)\rVert_{H^4}\le C\sqrt{M\epsilon^2}$ in \eqref{estimate_stokes}, we get
\begin{align}\label{spring2}
\rVert u_2\rVert_{H^3}\rVert \Delta^2\theta\rVert_{L^2}&\le_C\rVert u_2\rVert_{L^2}^{2/5}\rVert \Delta^2\Psi\rVert_{H^2}^{3/5}\sqrt{M\epsilon^2}\nonumber\\
& = \rVert u_2\rVert_{L^2}^{2/5}\sqrt{M\epsilon^2}g(t)^{3/10}\nonumber\\
& \le_C  \eta g(t)+  C_\eta\left(\sqrt{M\epsilon^2}\right)^{10/7}\rVert u_2\rVert_{L^2}^{4/7},\text{ for any $\eta>0$,}
\end{align}
where the last inequality follows from Young's inequality.
Similarly, the Gagliardo-Nirenberg interpolation theorem and the definition of $K(t)$ in \eqref{Stokes_potential} give us
\begin{align}\label{spring3}
\rVert \partial_1\theta\rVert_{L^2}^2=\rVert \Delta^2\Psi\rVert_{L^2}^2\le C\rVert \Delta^2\Psi\rVert_{H^2}\rVert \Delta \Psi\rVert_{L^2}\le \eta \rVert \Delta^2\Psi\rVert_{H^2}^2 + C_\eta \rVert \Delta\Psi\rVert_{L^2}^2\le_C \eta g(t) + C_\eta K(t).
\end{align}
Summing up this, \eqref{spring2} and \eqref{spring1}, we obtain
\begin{align*}
\left(\rVert u_2\rVert_{W^{2,\infty}} +\rVert u_2\rVert_{H^3} \right)f(t)& + \rVert u_2\rVert_{H^3}\rVert \Delta^2\theta\rVert_{L^2} + \rVert \partial_1\theta\rVert_{L^2}^2\\
&\le C\eta g(t)+C_\eta(\sqrt{M\epsilon^2})^{10/7}\rVert u_2\rVert_{L^2}^{4/7} + K(t).
\end{align*}
Substituting this into \eqref{rossgaller_stokes} and choosing $\eta$ small enough, we arrive at
\begin{align}\label{almostdone1}
\frac{d}{dt}f(t) +Cg(t)\le_C \left(\sqrt{M\epsilon^2}\right)^{10/7}\rVert u_2\rVert_{L^2}^{4/7}+K(t),\text{ for $t\in [0,T^*]$}.
\end{align}
 The estimate for $\rVert u_2\rVert_{L^2}$ in Corollary~\ref{Cor_Stokes} tells us that 
 \[
 \frac{2}{t}\int_{t/2}^{t} s^{1/4}\rVert u_2(s)\rVert_{L^2}^2ds\le C t^{1/4}\frac{2}{t}\int_{t/2}^t\rVert u_2(s)\rVert_{L^2}^2 ds\le C\frac{M\epsilon^2}{t^{15/4}},\text{ for all $t\in [0,T^*]$.}
 \]
Hence applying Lemma~\ref{fialmass} with $\alpha= 2/7$, $n=15/4$, we obtain
\[
\int_1^{T^*} t^{1/14}\rVert u_2(t)\rVert_{L^2}^{4/7}dt= \int_1^{T^*} \left(t^{1/4}\rVert u_2(t)\rVert_{L^2}^2\right)^{2/7}dt\le C \left(M\epsilon^2\right)^{2/7}.
\]
Thanks to \eqref{lmiddlet}, this implies 
\begin{align}\label{u2sdjsdj2}
\int_{T_M^*}^{T^*}\left(\sqrt{M\epsilon^2}\right)^{10/7}\rVert u_2(t)\rVert_{L^2}^{4/7}dt \le \left(\sqrt{M\epsilon^2}\right)^{10/7} (T_M^*)^{-1/14}\int_{T^*_M}^{T^*}t^{1/14}\rVert u_2(t)\rVert_{L^2}^{4/7}dt\le CM\epsilon^2(T^*_M)^{-1/14}.
\end{align}
Similarly, for the term $K(t)$ in \eqref{almostdone1}, we use Corollary~\ref{Cor_Stokes} to see
\[
\frac{2}{t}\int_{t/2}^{t}s^{1/14}K(s)ds \le C t^{1/14}\frac{2}{t}\int_{t/2}^{t} K(s)ds\le C\frac{M\epsilon^2}{t^{41/14}}.
\]
Applying Lemma~\ref{fialmass} with $\alpha=1$ and $n=41/14$, we get $\int_1^{T^*}t^{1/14}K(t)dt \le CM\epsilon^2$. 
Therefore
\[
\int_{T^*_M}^{T^*}K(s)ds\le \left(T^*_M\right)^{-1/14}\int_{T^*_M}^{T^*}s^{1/14} K(s)ds\le C(T^*_M)^{-1/14}M\epsilon^2.
\]
Combining this with \eqref{u2sdjsdj2}, integrating \eqref{almostdone1} over $t\in [T_M^*,T^*]$ gives 
\[
f(T^*) + \int_{T_M^*}^{T^*} g(s) ds\le_C  f(T^*_M) + \left( T^*_M\right)^{-1/14} M\epsilon^2
\]
Together with \eqref{middleestimagte} and \eqref{lmiddlet}, this estimate finally gives us
\[
f(T^*)+\int_{0}^{T^*}g(t)dt\le_CM\epsilon^2\left(M^{-1}(\log M)^C + (\log\log M)^{-1/14} \right).
\]
Comparing this to \eqref{reallyjoey_Stokes}, we must have
\[
1\le C\left(M^{-1}(\log M)^C + (\log\log M)^{-1/14} \right),
\]
which leads to a contradiction if $M$ is chosen sufficiently large depending only on the implicit constant $C$. This finishes the proof.

\begin{appendix}
\section{Proof of Lemma~\ref{onedirection}}\label{Proof_oflemma3_1}
We argue by induction on $k$, that is, we will prove that for every $k\in \mathbb{N}$, 
\begin{align}\label{induction_pre}
\rVert \partial_{1}^n\partial_2^kf\rVert_{L^2}\le C_{n,k}\left( \rVert \partial_1^{n+k}f\rVert_{L^2} +\rVert \partial_2^{n+k}f\rVert_{L^2}  \right),\text{ for all $n\in\mathbb{N}\cup\left\{ 0\right\}$ and $f\in X^\infty$.}
\end{align}
When $k=0$, the above estimate  holds true trivially. Now, let us aim to prove \eqref{induction_pre} for $k\to k+1$, assuming that the statement is true for some $k\ge0$.
We split the cases $n=0, 1$ and $n\ge 2$.

\textbf{Case: $n=0$.} It is trivial that
\begin{align}\label{neqzero}
\rVert \partial_2^{k+1}f\rVert_{L^2}\le C_{k}\left( \rVert \partial_1^{k+1}f\rVert_{L^2} +\rVert \partial_2^{k+1}f\rVert_{L^2}  \right).
\end{align}

\textbf{Case: $n=1$.} We have
\begin{align*}
\rVert \partial_1 \partial_2^{k+1}f\rVert_{L^2}^2 =\int_\Omega \partial_1\partial_2^{k+1}f(x)\partial_1\partial_2^{k+1}f(x)dx &= -\int_{\Omega}\partial_{11}\partial_2^{k+1}f(x)\partial_2^{k+1}f(x)dx = \int_{\Omega} \partial_{11}\partial_2^{k}f(x)\partial_{2}^{k+2}f(x)dx\\
&\le \rVert \partial_{11}\partial_2^kf\rVert_{L^2}^2 + \rVert \partial_2^{k+2}f\rVert_{L^2}^2,
\end{align*}
where the integration by parts to get the second equality is justified by \eqref{Xprop}. Using the induction hypothesis \eqref{induction_pre} for $k$, we have $ \rVert \partial_{11}\partial_2^kf\rVert_{L^2}^2\le C_k\left(\rVert \partial_1^{k+2}f\rVert_{L^2}^2 +\rVert \partial_2^{k+2}f\rVert_{L^2}^2 \right)$, hence
\begin{align}\label{neqone}
\rVert \partial_1 \partial_2^{k+1}f\rVert_{L^2}^2 \le C_k\left(\rVert \partial_1^{k+2}f\rVert_{L^2}^2 +\rVert \partial_2^{k+2}f\rVert_{L^2}^2 \right).
\end{align}

\textbf{Case: $n\ge 2$}. 
 We claim that
 \begin{align}\label{second_induction}
 \rVert \partial_{1}^{n-j}\partial_2^{k+1+j}f\rVert^2_{L^2(\Omega)}\le \frac{1}{2}\left(\rVert \partial_1^{n-j+1}\partial_2^{k+j}f\rVert_{L^2}^2  + \rVert\partial_1^{n-j-1}\partial_2^{k+j+2}f\rVert_{L^2}^2\right),\text{ for all $0\le j\le n-1$.}
 \end{align}
 Suppose \eqref{second_induction} holds true for the moment. Then  we have 
 \begin{align}\label{twro2}
 \sum_{j=0}^{n-1}  \rVert \partial_{1}^{n-j}\partial_2^{k+1+j}f\rVert^2_{L^2(\Omega)} &\le \frac{1}{2} \sum_{j=0}^{n-1}\rVert \partial_1^{n-j+1}\partial_2^{k+j}f\rVert_{L^2}^2  +\frac{1}{2} \sum_{j=0}^{n-1} \rVert\partial_1^{n-j-1}\partial_2^{k+j+2}f\rVert_{L^2}^2=:S_1+S_2
 \end{align}
 The first summation in the right-hand side can be written as
 \begin{align*}
 S_1  &= \frac{1}{2}\left(\rVert \partial_1^{n+1}\partial_2^{k}f\rVert_{L^2}^2 + \sum_{j=0}^{n-2}\rVert\partial_1^{n-j}\partial_2^{k+j+1}f\rVert_{L^2}^2\right)\\
 &= \frac{1}{2}\left(\rVert \partial_1^{n+1}\partial_2^{k}f\rVert_{L^2}^2 +\rVert \partial_1^n\partial_2^{k+1}f\rVert_{L^2}^2+ \sum_{j=1}^{n-2}\rVert\partial_1^{n-j}\partial_2^{k+j+1}f\rVert_{L^2}^2\right)
 \end{align*}
 and the second summation can be written as
 \begin{align*}
 S_2 &=\frac{1}{2}\left(\sum_{j=1}^{n-1}\rVert\partial_1^{n-j}\partial_2^{k+j+1}f\rVert_{L^2}^2 + \rVert \partial_2^{n+k+1}f\rVert_{L^2}^2\right)\\
 &=\frac{1}{2}\left(\sum_{j=1}^{n-2}\rVert\partial_1^{n-j}\partial_2^{k+j+1}f\rVert_{L^2}^2 + \rVert\partial_1^{1}\partial_2^{k+n}f\rVert_{L^2}^2 + \rVert \partial_2^{n+k+1}f\rVert_{L^2}^2 \right)
  \end{align*}
 Summing them up,  we can write the right-hand side of \eqref{twro2}  as
 \begin{align*}
 S_1+S_2 = \sum_{j=1}^{n-2}  \rVert \partial_{1}^{n-j}\partial_2^{k+1+j}f\rVert^2_{L^2(\Omega)} +\frac{1}{2}\left(\rVert \partial_1^{n+1}\partial_2^{k}f\rVert_{L^2}^2 +\rVert \partial_1^n\partial_2^{k+1}f\rVert_{L^2}^2 + \rVert\partial_1\partial_2^{k+n}f\rVert_{L^2}^2 + \rVert \partial_2^{n+k+1}f\rVert_{L^2}^2 \right)
 \end{align*}
 Plugging this into \eqref{twro2}, and after the cancellations of the summation in $j$ over $1$ to $n-2$, we obtain
 \[
 \rVert\partial_1^n\partial_2^{k+1}f\rVert_{L^2}^2+\rVert \partial_1\partial_2^{k+n}f\rVert_{L^2}^2\le \frac{1}{2}\left(\rVert \partial_1^{n+1}\partial_2^{k}f\rVert_{L^2}^2 +\rVert \partial_1^n\partial_2^{k+1}f\rVert_{L^2}^2 + \rVert\partial_1\partial_2^{k+n}f\rVert_{L^2}^2 + \rVert \partial_2^{n+k+1}f\rVert_{L^2}^2 \right),
 \]
 which is equivalent to
 \[
  \rVert\partial_1^n\partial_2^{k+1}f\rVert_{L^2}^2+\rVert \partial_1\partial_2^{k+n}f\rVert_{L^2}^2 \le \rVert \partial_1^{n+1}\partial_2^{k}f\rVert_{L^2}^2  +  \rVert \partial_2^{n+k+1}f\rVert_{L^2}^2.
 \]
 Hence, dropping the second term in the left-hand side, we get
 \[
   \rVert\partial_1^n\partial_2^{k+1}f\rVert_{L^2}^2\le  \rVert \partial_1^{n+1}\partial_2^{k}f\rVert_{L^2}^2  +  \rVert \partial_2^{n+k+1}f\rVert_{L^2}^2\le C_{n,k}\left( \rVert \partial_1^{n+k+1}f\rVert_{L^2}^2+ \rVert \partial_2^{n+k+1}f\rVert_{L^2}^2\right),
 \]
 where we used the induction hypothesis \eqref{induction_pre} for $k$. Combining this with \eqref{neqzero} and \eqref{neqone}, we see that the statement \eqref{induction_pre} holds true for all $n$, for $k+1$. The only missing part is a proof of \eqref{second_induction}.

\textbf{Proof of \eqref{second_induction}.}
  It is straightforward that
  \begin{align*}
   \rVert \partial_{1}^{n-j}\partial_2^{k+1+j}f\rVert^2_{L^2(\Omega)}&= \int_{\Omega} \partial_{1}^{n-j}\partial_2^{k+1+j}f(x) \partial_{1}^{n-j}\partial_2^{k+1+j}f(x)dx\\
   & =-\int _{\Omega} \partial_1\partial_{1}^{n-j}\partial_2^{k+1+j}f(x) \partial_{1}^{n-j-1}\partial_2^{k+1+j}f(x)dx\\
   & = \int_{\Omega} \partial_{1}^{n-j+1}\partial_2^{k+j}f(x) \partial_{1}^{n-j-1}\partial_2^{k+2+j}f(x)dx\\
   &\le \frac{1}{2}\rVert\partial_1^{n-j+1}\partial_2^{k+j}f\rVert_{L^2}^2 + \frac{1}{2}\rVert\partial_1^{n-j-1}\partial_2^{k+2+j}f\rVert_{L^2}^2, 
  \end{align*}
  where the second and the third equality follow from the integration by parts which is valid due to \eqref{Xprop} and the last inequality follows from Young's inequality. This proves the claim \eqref{second_induction}.
\end{appendix}

 \bibliographystyle{abbrv}
\bibliography{references}

\end{document}